\documentclass[a4paper]{amsart}
\usepackage[english]{babel}
\usepackage[utf8]{inputenc}
\usepackage{fancyhdr,setspace,amsmath,amsthm,amsfonts}
\usepackage{amsmath}
\usepackage{amssymb }
\usepackage{bbm}
\usepackage{graphicx}
\usepackage{caption}
\usepackage{subcaption}
\usepackage{float}
\usepackage[colorinlistoftodos]{todonotes}
\usepackage{enumerate}
\newtheorem{theorem}{Theorem}[section]

\newtheorem{lemma}[theorem]{Lemma}
\newtheorem{proposition}[theorem]{Proposition}
\newtheorem{definition}[theorem]{Definition}
{
      \theoremstyle{plain}
      
  }

\def\enddisplaymath{\]\@ignoretrue}

\title{Phase Transitions of the Multifractal Spectrum}

\author{Jason Tomas Dungca}

\date{\today}

\begin{document}

\begin{abstract}
We consider the multifractal analysis of the pointwise dimension for Gibbs measures on countable Markov shifts.  Our paper analyses the set of non-analytic points or phase transitions of the multifractal spectrum.  By Sarig's thermodynamic formalism for countable Markov shifts and Iommi's expression of the multifractal spectrum, we apply analyticity arguments on the pressure function for the countable shift.  Finally, we apply our results about the phase transitions of the multifractal spectrum to the Gauss map.
\end{abstract}

\maketitle
\section{Introduction}
Multifractal analysis is the study of the concentration of a measure on level sets.  In particular, we use measures for a hyperbolic map to study the multifractal spectrum.   Multifractal analysts historically modeled an expanding map with a finite state shift.  Their main result is the expression of the multifractal spectrum, which is concave and analytic.  However, the multifractal spectrum is not analytic everywhere in the case of countable state Markov shifts.  Iommi \cite{iommi2005multifractal} proves an expression for the multifractal spectrum in the setting of a countable Markov shift.  Using Iommi's result, our paper proves that the multifractal spectrum has up to infinitely many phase transitions or non-analytic points. 
\par Before outlining our results, we give some background and definitions in the area of multifractal analysis.  Consider the countable state shift $\Sigma$ satisfying topological mixing and the Big Images Property.  We take the measure $\mu$ on our countable shift to be a Gibbs state and $\alpha \in \mathbb{R}.$  The sequence $x$ is said to have pointwise dimension $\alpha$ if     
\[
d_{\mu}(x):=\lim\limits_{m \rightarrow \infty}\frac{\log\mu([x_{1},...,x_{m}])}{\log|[x_{1},...,x_{m}]|}=\alpha.\]
Let the set of sequences $x$ in $\Sigma$ such that  $d_{\mu}(x)$ does not exist be denoted as $X'$ and let the set of sequences with pointwise dimension $\alpha$ be denoted as $X_{\alpha}^{s}.$  Hence, we can decompose $\Sigma$ into a disjoint union called the multifractal decomposition based on the pointwise dimension of each $x$ in our countable shift.  We will consider a function that gives the Hausdorff dimension of sets that have a pointwise dimension.  The multifractal spectrum with respect to $\mu$ is the function \[f_{\mu}(\alpha)=\dim_{H}(X_{\alpha}^{s}).\]
We remark that the multifractal spectrum is similarly defined for finite state Markov shifts. 
 \par We first discuss historical results about the multifractal spectrum in the setting of finite state Markov shifts.  Rand \cite{rand1989singularity} considers a cookie-cutter, which is a uniformly hyperbolic map.  He uses thermodynamic formalism on the finite state Markov shift to prove that the multifractal spectrum is everywhere analytic.  Next, we discuss the work of Cawley and Mauldin \cite{cawley1992multifractal}.  They consider a fractal constructed by taking an iterated function system on a self-similar set (which yields a self-similar measure).  Falconer \cite{falconer2004fractal} gives more details on the construction of such a fractal.  Cawley and Mauldin essentially model an iterated function system, based on countably many contractions, with a finite state Markov shift.  Their methodology to prove that the multifractal spectrum is everywhere analytic involves geometric arguments. \par
Pesin and Weiss \cite{pesin1997multifractal} consider a uniformly expanding map and prove an expression for the multifractal spectrum.  They use a combination of thermodynamic formalism and a covering argument to prove that the multifractal spectrum is analytic everywhere.  Analysing the multifractal spectrum involves using thermodynamic formalism differently in the case of measures on the countable full shift $\Sigma.$  Iommi \cite{iommi2005multifractal} proves a general formula for the multifractal spectrum with respect to a measure on the countable shift and analyses when the multifractal spectrum is non-analytic.  Hanus, Mauldin, and Urba\'nski also consider the multifractal spectrum in the setting of the countable conformal iterated function system modeled by a countable Markov shift.  Complementary to Iommi's result about the multifractal spectrum's non-analyticity, their paper \cite{hanus2002thermodynamic} gives additional conditions to prove that the multifractal spectrum is analytic.  
\par    
Sarig has developed the thermodynamic formalism on countable Markov shifts $\Sigma.$  His work \cite{sarig2003existence} has established criteria for the existence of Gibbs and equilibrium states for potentials.  This is significant because the existence of equilibrium states, which are important for topological pressure, is not guaranteed for potentials on countable Markov shifts.  These measures are critical in our paper to prove results about the non-analyticity points of $f_{\mu}(\alpha)$ in the setting of a countable Markov shift.  For a thorough discussion of Sarig's work, please refer to Sarig's survey \cite{sarig2015thermodynamic}. \par
We remark that Iommi and Jordan's work analyses the phase transitions of the pressure function and the multifractal spectrum in a similar setting to our paper.  Their paper \cite{iommi2015multifractal}, assumes that the two potentials $\phi$ and $\psi$ are bounded and $\lim\limits_{x \rightarrow 0}\frac{\psi(x)}{\log|T'(x)|}=\infty$ (in contrast, we assume that the limit is equal to $1$).  Their result is that the multifractal spectrum has $0$ to $2$ phase transitions.  In the paper \cite{iommi2013phase}, they take $g$ to be a continuous function defined on the range of the suspension flow.  Then, they prove that the map $t \mapsto \mathcal{P}(tg)$ has $0$ to $1$ phase transition when the roof function dominates the floor function.  Using these results, Iommi and Jordan prove that the multifractal spectrum has $0$ to $2$ phase transitions in their paper \cite{iommi2013phase}. \par  
Their paper \cite{iommi2010multifractal}, considers level sets of $[0,1]$ generated by Birkhoff averages and expanding interval maps that have countably many branches.  This paper uses results from \cite{iommi2013phase} and \cite{iommi2015multifractal}.  Their results include a variational characterisation of the multifractal spectrum and the existence of $0$ to $2$ phase transitions for $f_{\mu}(\alpha)$ when $\alpha_{\lim},$ which is a ratio involving the potentials $\phi$ and $\psi,$ equals $0.$  In contrast, our paper assumes that $0 < \alpha_{\lim} \le \infty$ or does not exist.  Their paper also proves results on the multifractal analysis of suspension flows.  Let $T$ be an expanding interval map and $g$ be a continuous function defined on the range of the suspension flow.  Iommi and Jordan prove that the Birkhoff spectrum with respect to $g$ has two phase transitions if the roof function dominates the geometric potential $\log|T'|.$  \par

We briefly discuss different examples of the multifractal spectrum's phase transitions in other settings.  Researchers have studied phase transitions for non-uniformly expanding interval maps that have neutral fixed points.  They use thermodynamical formalism and they respectively prove explicit formulae for the multifractal spectrum.  Olivier's paper \cite{olivier2000structure} considers a cookie cutter on $[0,1]$ and in turn, takes an induced map defined on a Cantor set generated by this cookie cutter.  Nakaishi's paper \cite{nakaishi2000multifractal} considers piecewise interval maps on $[0,1]$ and an induced transformation generated by these maps.  We note that Nakaishi's paper is related to a paper by Pollicott and Weiss \cite{pollicottmultifractal}.  The multifractal spectrum with respect to Bernoulli convolutions for algebraic paramters has been analysed in an example in Feng and Olivier's paper \cite{feng2003multifractal}.

Now that we have given some background into multifractal analysis, we outline our paper's results and methodology to obtain these results.  This paper is about multifractal analysis in the setting of the countable Markov shift $\Sigma$ and the interval $[0,1].$  In our analysis of $f_{\mu}(\alpha),$ we use thermodynamic formalism.  We take the locally H\"older potential functions $\phi:\Sigma \rightarrow \mathbb{R}^{-}$ with Gibbs measure $\mu$ and the metric potential $\psi:\Sigma \rightarrow \mathbb{R}^{+}$ such that $\psi=\log|G'|$ for some expanding map $G.$ 

We use results from Sarig about Gibbs states on countable Markov shifts to prove the existence of Gibbs states for potential functions related to $\phi$ and $\psi.$  These Gibbs states help us show that the multifractal spectrum has phase transitions.  Proving the following theorem about the analyticity of the multifractal spectrum requires using results from Sarig, Mauldin, Urba\'nski, and Iommi.  We define $\alpha_{\lim}$ as follows:
Let $\bar{i} \in \Sigma$ be such that $\bar{i}=(i,i,...).$  Take
\[\alpha_{\lim}:=\lim\limits_{i \rightarrow \infty}\frac{\phi(\bar{i})}{-\psi(\bar{i})}\]
if the limit exists.  We also let $\alpha_{\inf}:=\inf\{d_{\mu}(x): x \in \Sigma\}$ and $\alpha_{\sup}:=\sup\{d_{\mu}(x): x \in \Sigma\}.$  We remark that we will introduce a different definition of $\alpha_{\inf}$ and $\alpha_{\sup}$ in our paper and then prove the equivalence of both definitions of $\alpha_{\inf}$ and $\alpha_{\sup}.$

\begin{theorem}\label{thm:M1}
Let $\phi: \Sigma \rightarrow \mathbb{R}^{-}$  be a potential with Gibbs measure $\mu$ such that $P(\phi)<\infty.$ and $\psi: \Sigma \rightarrow \mathbb{R}^{+}$ be a metric potential.  Assume that $\phi$ and $\psi$ are non-cohomologous locally H\"older potentials such that $\alpha_{\lim}<\infty.$   
\begin{enumerate}
\item There exist intervals $A_{i}$ such that $f_{\mu}(\alpha)$ is analytic on each of their interiors.  
\item The interval $(\alpha_{\inf}, \alpha_{\sup})=\cup_{i=1}^{j}A_{i}$ such that $j=\{1,2,3,4\}.$ 
\item The multifractal spectrum is concave on $(\alpha_{\inf}, \alpha_{\sup}),$ has its maximum at $\alpha(0),$ and has zero to three phase transitions.
\end{enumerate}
\end{theorem}

We prove a theorem about the behaviour of the multifractal spectrum when $\alpha_{\lim}$ is infinite.  Since some of the earlier results about the analyticity of the mutlfractal spectrum are still true for $\alpha_{\lim}=\infty,$ we prove a theorem stating that the multifractal spectrum has $0$ to $1$ phase transition as provided below.

\begin{theorem}\label{thm:M2}
Let $\phi: \Sigma \rightarrow \mathbb{R}^{-}$  be a potential with Gibbs measure $\mu$ such that $\mathcal{P}(\phi)<\infty$ and $\psi: \Sigma \rightarrow \mathbb{R}^{+}$ be a metric potential.  Assume that $\phi$ and $\psi$ are non-cohomologous locally H\"older potentials such that $\alpha_{\lim}=\infty.$   
\begin{enumerate}
\item There exist intervals $A_{i}$ such that $f_{\mu}(\alpha)$ is analytic on each of their interiors.  
\item The interval $(\alpha_{\inf}, \alpha_{\sup})=\cup_{i=1}^{j}A_{i}$ such that $j=\{1,2\}.$  
\item The multifractal spectrum is concave on $(\alpha_{\inf}, \alpha_{\sup}),$ is equal to its maximum $f_{\mu}(\alpha(0))$ on $(\alpha(0),\alpha_{\sup}),$ and has zero to one phase transition.
\end{enumerate}
\end{theorem}

Finally, we apply Theorem \ref{thm:M1} and Theorem \ref{thm:M2} to the Gauss map $G$ by defining our locally H\"older potential as $\psi(x)=\log|G'(\pi(x))|$ with respect to the coding map $\pi:\Sigma \rightarrow [0,1].$  We provide examples that apply Theorems \ref{thm:M1} and \ref{thm:M2} to show that the multifractal spectrum has up to three phase transitions.  Also, we provide an example in which the multifractal spectrum has infinitely many phase transitions.  In these examples, we will use locally H\"older potentials to estimate $\phi$ and $\psi.$  Now, we will define some notation and concepts from thermodynamic formalism.

\subsection*{Acknowledgements}\hspace*{\fill} \par
The author thanks Dr.\ Thomas Jordan  for his direction and guidance in the writing of this paper and Dr.\ Godofredo Iommi for his helpful comments.

\section{Thermodynamic Formalism}
We now introduce some concepts from thermodynamic formalism.

\subsection{Introductory Definitions from Thermodynamic Formalism}\hspace*{\fill} \par
Before proving our result about the phase transitions of the multifractal spectrum, we introduce some definitions from thermodynamic formalism.  Denote $S$ as our countable state space and $A=(a_{lm})_{SxS}$ as our transition matrix of zeroes and ones.  We will be treating $S$ as $\mathbb{N}.$  Note that $A$ can be represented by a directed graph.  We let $$\Sigma_{A}:=\{x \in S^{\mathbb{N}}: a_{x_{i}x_{i+1}}=1 \text{ for every } i \ge 1\}.$$
We take $\sigma: \Sigma_{A} \rightarrow \Sigma_{A}$ to be the standard left shift.  We now define the topology for our countable state Markov shift $\Sigma_{A}.$

\begin{definition}
Given $x_{1}, ..., x_{n}$ symbols in $S,$ define a {\em cylinder set} in $\Sigma_{A}$ as
$$[x_{1},...,x_{n}]=\{y \in \Sigma_{A}: y_{i}=x_{i}  \text{ for } 1 \le i \le n\}.$$
\end{definition}

These cylinder sets form the topology for $\Sigma_{A}.$  Two important assumptions for our countable Markov shift are defined below.
\begin{definition}
The shift space $\Sigma_{A}$ satisfies the {\em big images and pre-images property} if there is a finite set $\{c_{1},c_{2},...,c_{m}\}$ from our alphabet $S$ such that for each $d \in S,$ there are $i,j \in \{1,...,m\}$ such that
$$a_{c_{i}d}a_{dc_{j}}=1.$$
\end{definition}

\begin{definition}
$\sigma:\Sigma_{A} \rightarrow \Sigma_{A}$ is said to be {\em topologically mixing} if for all $a,b \in S,$ there exists $N_{ab} \in \mathbb{N}$ such that for all $n>N_{ab},$
$$[a] \cap \sigma^{-n}[b] \not= \emptyset.$$
\end{definition}

We provide the following remark.  Since $\Sigma_{A}$ is a non-compact shift space, the existence of topological mixing alone for $\Sigma$ is not enough for a Gibbs measure to exist.  For this reason, Sarig \cite{sarig2003existence} proves that the combination of topologically mixing and the BIP (big images and pre-images property) is sufficient and necessary for a Gibbs measure to exist.  Denote $\Sigma_{A}$ as $\Sigma.$  We assume that $\Sigma$ satisfies the big images and pre-images property and it is topological mixing.  Now, we must define a property for functions on $\Sigma.$
\begin{definition}
Let $\psi: \Sigma \rightarrow \mathbb{R}.$  The {\em$n$-th variation} of $\psi$ is given by
$$V_{n}(\psi):=\sup\limits_{[x_{1},...,x_{n}] \subset \Sigma_{A}}\sup\limits_{x,y \in [x_{1},...,x_{n}]}|\psi(x)-\psi(y)|.$$
\end{definition}

\begin{definition}
Let $\psi: \Sigma \rightarrow \mathbb{R}.$  $\psi$ is said to be {\em locally H\"older continuous} if there exists $C>0$ and $\theta \in (0,1)$ such that for each $n \in \mathbb{N},$
$$V_{n}(\psi) \le C\theta^{n}.$$
\end{definition}

Local H\"older continuity is important in the definition of our metric for $\Sigma.$  Consider the following metric $d$ on the countable shift $\Sigma.$  Take the sequences $x=(x_{1},x_{2},...), y=(y_{1},y_{2},...) \in \Sigma.$  Find the first common, starting subword in which $x$ and $y$ agree. \par 
In other terms, take $x \wedge y=x_{1},...,x_{k}$ such that $k=\max\{m \in \mathbb{N}:x_{i}=y_{i} \text{ for all } 1 \le i \le m\}.$  For each $i \in \mathbb{N},$ let $0<r_{i}<1.$
Then, given $x$ and $y,$ we define our metric as
$$d(x,y)=r_{x_{1}} \cdots r_{x_{k}}.$$

We can define a potential $\psi$ connected to this metric.  In many of our examples, we will define a potential $\psi: \Sigma \rightarrow \mathbb{R}^{+}$ as a locally constant function:
\[\psi(x)=\log{r_{x_{1}}^{-1}}=-\log{r_{x_{1}}}.\]

Then,
\begin{equation}\label{eq:B}
0 \le \prod_{j=0}^{n-1}(\exp(\psi(\sigma^{j}(x))))^{-1}=(r_{x_{1}}^{-1}r_{x_{2}}^{-1}\cdots r_{x_{n}}^{-1})^{-1}=r_{x_{1}}r_{x_{2}} \cdots r_{x_{n}}.
\end{equation}

Given our metric,
\begin{equation}\label{eq:C}
|[x_{1},...,x_{n}]|=\sup\limits_{x,y \in [x_{1},...,x_{n}]}d(x,y)=r_{y_{1}}r_{y_{2}} \cdots r_{y_{n}}.
\end{equation}

Therefore, by Equations (\ref{eq:B}) and (\ref{eq:C}),
\[\prod_{j=0}^{n-1}(\exp(\psi(\sigma^{j}(x))))^{-1} =r_{x_{1}}r_{x_{2}} \cdots r_{x_{n}}.\]
In fact, the equation above relates the diameter of the cylinder set $[x_{1},...,x_{n}]$ to the potential $\psi.$  Potentials satisfying such an inequality are called metric potentials, which are defined as follows. 

\begin{definition}
A positive, locally H\"older potential $\psi$ with the property that there exists a constant $C>0$ such that
$$\frac{1}{C}\prod_{j=0}^{n-1}(\exp(\psi(\sigma^{j}(y))))^{-1} \le |[x_{1},x_{2},...,x_{n}]| \le C\prod_{j=0}^{n-1}(\exp(\psi(\sigma^{j}(y))))^{-1},$$
for every $y=(x_{1},x_{2},...,x_{n},y_{n+1},...) \in [x_{1},...,x_{n}],$ is said to be a {\em metric potential}.
\end{definition}
 
Thus, $\psi(x)=\log{r_{x_{1}}^{-1}}$ is a metric potential with respect to the chosen metric $d$ on $\Sigma$ and the constant $C=1.$  We can define more general metric potentials on different metrics on $\Sigma.$  We remark that it is important for $\psi$ to be a metric potential defined by a hyperbolic map.  For instance, take a general expanding map $T:[a,b] \rightarrow [a,b]$ for $b > a \ge 0$ and then, assuming that $\psi(x)=\log|T'(\pi(x))|,$ we find that
\[-C\exp\left(\sum_{j=0}^{n-1}-\log|T'(\pi(x))|\right) \le |[x_{1},x_{2},...,x_{n}]| \le C\exp\left(\sum_{j=0}^{n-1}-\log|T'(\pi(x))|\right).\]
Hence, given the expanding map $T$ and metric potential $\psi,$ we are able to relate the diameter of cylinder sets on $\Sigma$ to the radii of balls on an interval.  Defining our metric by using our potential also gives an additional expression for the local dimension of sets in $\Sigma.$ Furthermore, we would be able to calculate the Hausdorff dimension of sets with local dimension $\alpha,$ so we must give the following definition.

\begin{definition}
Let $\psi: \Sigma \rightarrow \mathbb{R}$ be locally H\"older.  We define the {\em topological pressure} as follows
$$\mathcal{P}(\psi)= \lim\limits_{n \rightarrow \infty}\frac{1}{n}\log\sum_{[x_{1},...,x_{n}] \subset \Sigma}\exp\sup\limits_{y \in [x_{1},...,x_{n}]}\left(\sum_{i=0}^{n-1}\psi(\sigma^{i}(y))\right).$$
\end{definition}

We also define another form of pressure.

\begin{definition}
Let $\psi: \Sigma \rightarrow \mathbb{R}$ be locally H\"older.  We define the {\em Gurevich pressure} as follows
$$\mathcal{P}_{G}(\psi)= \lim\limits_{n \rightarrow \infty}\frac{1}{n}\log\sum_{\sigma^{n}x=x}\exp\left(\sum_{i=0}^{n-1}\psi(\sigma^{i}(x))\right)\mathbbm{1}_{[x_{1}]}$$
such that $\mathbbm{1}_{[x_{1}]}$ is the indicator function on the cylinder $[x_{1}].$
\end{definition}

Because $\Sigma$ is topologically mixing, the Gurevich pressure does not depend on $x_{1}$.  Sarig proves that the Gurevich pressure is equivalent to topological pressure when $\Sigma$ is BIP.  We provide this result by Sarig \cite{sarig1999thermodynamic} (Pg 1571, Theorem 3), Mauldin, and Urba\'nski \cite{mauldin2003graph} (Pg 11, Theorem 2.1.8) below.

\begin{proposition}
Let $\Sigma$ be topologically mixing and $\psi:\Sigma \rightarrow \mathbb{R}$ be locally H\"older such that $\sup\psi<\infty.$ Let $M_{\sigma}(\Sigma)$ be the set of $\sigma-$invariant measures.  Then,
\[\mathcal{P}(\psi)=\sup_{\mu \in M_{\sigma}(\Sigma)}\left\{\int \phi\, \mathrm{d}\mu + h(\mu)\right\}=\mathcal{P}_{G}(\psi).\]
\end{proposition}

We remark that Iommi, Jordan, and Todd \cite{iommijordantodd} (Pg 8, Theorem 2.10) proved that $\sup\psi<\infty$ is an unnecessary condition for the previous proposition.  We will provide a result that gives another way to calculate the topological pressure of a potential.  Let $\psi$ and $\psi$ be locally H\"older.  We can approximate the topological pressure of a potential $q\phi-t\psi$ (with $q,t \in \mathbb{R}$) on the shift $\Sigma$ by restricting the potential to a compact, invariant subset $K \subset \Sigma.$  For such a set $K,$ denote 
$$\mathcal{P}_{K}(q\phi-t\psi):=P((q\phi-t\psi)|_{K})$$
as the restriction of the topological pressure of $q\phi-t\psi$ to $K.$  Sarig \cite{sarig1999thermodynamic} (Pg 1570, Theorem 2), Mauldin, and Urba\'nski \cite{mauldin2003graph} (Pg 8, Theorem 2.1.5) have proven the following proposition.
\begin{proposition}
Let $\phi$ and $\psi$ be locally H\"older.  If $\mathcal{K}:=\{K \subset \Sigma: K \text{ compact and }\\ 
\sigma\text{-invariant, } K \not= \emptyset\},$ then
$$\mathcal{P}(q\phi-t\psi)=\sup\limits_{K \in \mathcal{K}}\mathcal{P}_{K}(q\phi-t\psi).$$
\end{proposition}

Compact subsets of our countable Markov shift $\Sigma$ include finite state Markov shifts $\Sigma_{n}.$  We will later use a nested sequence of finite state full shifts to approximate the topological pressure of $q\phi-t\psi$ on $\mathbb{N}^{\mathbb{N}}.$

\begin{definition}
A probability measure $\mu$ is said to be a {\em Gibbs measure} for the potential $\phi: \Sigma \rightarrow \mathbb{R}$ if there exist two constants $M$ and $P$ such that, for each cylinder $[x_{1},x_{2},...,x_{m}]$ and every $x \in [x_{1},x_{2},...,x_{m}],$
$$\frac{1}{M} \le \frac{\mu([x_{1},x_{2},...,x_{m}])}{\exp(-mP+\sum_{j=0}^{m-1}\phi(\sigma^{j}(x)))} \le M.$$
\end{definition}

In fact, Mauldin and Urba\'nski \cite{mauldin2003graph} (Pg 13, Proposition 2.2.2) proved that $P=\mathcal{P}(\phi).$  Now, we will define the locally H\"older potentials needed for our analysis of the multifractal spectrum.  

\subsection{The Potentials $\phi$ and $\psi$}\hspace*{\fill} \par
Throughout this paper, let $\phi: \Sigma \rightarrow \mathbb{R}^{-}$ be a locally H\"older potential such that $0 \le \mathcal{P}(\phi)<\infty.$  We assume that $\mu$ is the Gibbs measure for $\phi.$  We will show the existence of $\mu$ later.  Also, let $\psi: \Sigma \rightarrow \mathbb{R}^{+}$ be a locally H\"older, metric potential with respect to an appropriate metric.

\begin{definition}
Two functions $\phi:\Sigma \rightarrow \mathbb{R}^{-}$ and $\psi:\Sigma \rightarrow \mathbb{R}^{+}$ are {\em cohomologous} in a class $\mathcal{H}$ if there exists a function $u: \Sigma \rightarrow \mathbb{R}$ in the class $\mathcal{H}$ such that
$$\phi-\psi=u-u\circ{\sigma}.$$
\end{definition}

We assume that $\phi$ and $\psi$ are non-cohomologous to each other.  For brevity, we will instead state that $\phi$ and $\psi$ are non-cohomologous.  For $q, t \in \mathbb{R},$ consider the family of potentials $q\phi-t\psi.$  Given this family of potentials, we will analyse the behaviour and phase transitions of a function $T(q)$ dependent on $q.$  This analysis is needed for our result about the multifractal spectrum's phase transitions.

\subsection{The Functions $T(q)$ and $\tilde{t}(q)$ and the Limit $\alpha_{\lim}$}\hspace*{\fill} \par
We will later prove that the multifractal spectrum's phase transitions are closely related to the behaviour of a family of potentials $q\phi-t\psi.$  To follow this argument, we must define the following function.
\begin{definition}
For each $q \in \mathbb{R},$ the {\em temperature function} $T(q)$ is defined as
$$T(q):= \inf\{t \in \mathbb{R}:\mathcal{P}(q\phi-t\psi) \le 0\}$$
\end{definition}

Proposition 4.3 on Pg 1892 of Iommi \cite{iommi2005multifractal} states the following important result.
\begin{proposition}\label{prop:CD1}
$T(q)$ is a convex and decreasing function.
\end{proposition}

We also define a function $\tilde{t}(q)$ which is similar to $T(q).$

\begin{definition}
For each $q \in \mathbb{R},$ the {\em function} $\tilde{t}(q)$ is defined as
$$\tilde{t}(q):=\inf\{t \in \mathbb{R}:\mathcal{P}(q\phi-t\psi) < \infty\}.$$
\end{definition}

We will explain the significance of the family of potentials $q\phi-\tilde{t}(q)\psi$ later.  We need the following limit to get an expression for $\tilde{t}(q).$

\begin{definition}\label{def:lim}
Let $\bar{i} \in \Sigma$ be such that $\bar{i}=(i,i,...).$  Let
\[\alpha_{\lim}=\lim\limits_{i \rightarrow \infty}\frac{\phi(\bar{i})}{-\psi(\bar{i})}\]
if the limit exists.
\end{definition}

Assume that
\[t_{\infty}:=\inf\{t \in \mathbb{R}: \mathcal{P}(-t\psi)<\infty\}\]
exists and it is finite.  When $\alpha_{\lim}$ exists, we will prove that
$$\tilde{t}(q)=-\alpha_{\lim}q+t_{\infty}.$$

\begin{definition}
Let $f:\Sigma \rightarrow \mathbb{R}$ be a potential.  We define the \text{\em nth partition function} to be $$Z_{n}(f)=\sum_{[x_{1},...,x_{n}] \subset \Sigma}\exp\sup\limits_{y \in [x_{1},...,x_{n}]}\left(\sum_{i=0}^{n-1}f\left(\sigma^{i}(y)\right)\right).$$
\end{definition}

The following lemma is a modified version of Proposition 2.1.9 on Pg 11 in Mauldin and Urba\'nski \cite{mauldin2003graph}.

\begin{lemma}\label{lemma:MU1}
If $\Sigma$ has the BIP property and $f:\Sigma \rightarrow \mathbb{R}$ is locally H\"older, then $\mathcal{P}(f)<\infty$ if and only if $Z_{1}(f)<\infty.$
\end{lemma}

We use this lemma in the proof of the following proposition.

\begin{proposition}\label{prop:L1}
Let $\phi: \Sigma \rightarrow \mathbb{R}^{-}$ and $\psi: \Sigma \rightarrow \mathbb{R}^{+}$ be non-cohomologous locally H\"older potentials.  Assume that $\alpha_{\lim}<\infty.$  Let $t_{\infty}=\inf\{t \in \mathbb{R}:\mathcal{P}(-t\psi)<\infty\}$ be finite.  Then,  
$$\tilde{t}(q)=-\alpha_{\lim}q+t_{\infty}.$$
\end{proposition}

\begin{proof}
Take $N \in \mathbb{N}$ large and an arbitrary $i \in \mathbb{N}.$  Consider the one periodic sequences $\bar{n}=(n,n,...)$ for $n>N$ and $\bar{i}=(i,i,...).$  Given $\phi$ and $\psi$ locally H\"older potentials, we obtain the following estimate if $x \in [i]:$
\begin{equation}\label{eq:Variation}
\phi(\bar{i}) \le \phi(x)+V_{1}(\phi) \text{ and }
\psi(\bar{i}) \le \psi(x)+V_{1}(\psi)
\end{equation}
such that $V_{1}(\phi), V_{1}(\psi) \ge 0.$  Since $\Sigma$ has the BIP property and $\phi$ and $\psi$ are locally H\"older, Lemma~\ref{lemma:MU1} gives us that 
$$Z_{1}(q\phi-t\psi)<\infty \text{ if and only if } \mathcal{P}(q\phi-t\psi)<\infty.$$

We have that for every $x=(x_{1},x_{2},...) \in \Sigma,$
\[Z_{1}(q\phi-t\psi)=\sum_{x_{1}=1}^{\infty}\exp\sup\limits_{x\in[x_{1}]}(q\phi-t\psi)(x).\]

By Equation (\ref{eq:Variation}), we get that for each $\bar{i}=(i,i,...)$ and $x=(i,x_{2},...),$
\[\sup\limits_{x\in[i]}(q\phi-t\psi)(x) \le (q\phi-t\psi)(\bar{i})+V_{1}(q\phi-t\psi).\]

Hence, our calculations for $\bar{i}$ above give us that
$$Z_{1}(q\phi-t\psi) \le \sum_{i=1}^{\infty}\exp((q\phi-t\psi)(\bar{i}))+V_{1}(q\phi-t\psi).$$

Fix an arbitrary $q \in \mathbb{R}.$  Thus, to prove that $Z_{1}(q\phi-t\psi)<\infty$ for some $t \in \mathbb{R},$ it suffices to prove that
\begin{equation}\label{eq:3}
\sum_{i=1}^{\infty}\exp((q\phi-t\psi)(\bar{i}))<\infty.
\end{equation}
Let $\varepsilon>0$ and take our one periodic sequence $\bar{n}$ introduced earlier.  By definition~\ref{def:lim}, we immediately get that
\begin{equation}\label{eq:4}
\frac{\phi(\bar{n})}{-\psi(\bar{n})}-\varepsilon \le \alpha_{\lim} \le \frac{\phi(\bar{n})}{-\psi(\bar{n})}+\varepsilon.
\end{equation}

If $\mathcal{P}(-t\psi)<\infty,$ then $Z_{1}(-t\psi) \le \sum_{i=1}^{\infty}\exp((-t\psi)(i)+V_{1}(-t\psi))<\infty$
by our estimates above and Lemma~\ref{lemma:MU1}.  
Furthermore, $\sum_{n=N}^{\infty}\exp((-t\psi)(\bar{n}))<\infty$ if and only if $\sum_{i=1}^{\infty}\exp((-t\psi)(\bar{i}))<\infty.$  To prove Inequality~\ref{eq:3} for $\bar{i},$ we will show that \[\sum_{n=1}^{\infty}\exp((q\phi-t\psi)(\bar{n}))<\infty\]
for our fixed $q \in \mathbb{R}$ and some $t \in \mathbb{R}.$

When $t=t_{\infty}-q\alpha_{\lim}$ and $\varepsilon \rightarrow 0,$
\begin{eqnarray}
\sum_{n=N}^{\infty}\exp((q\phi-t\psi)(\bar{n})) &=& \sum_{n=N}^{\infty}\exp((q\phi+-(t_{\infty}+q(\frac{\phi(\bar{n})}{\psi(\bar{n})}+\varepsilon)\psi)(\bar{n})) \nonumber\\
&=&\sum_{n=N}^{\infty}\exp((-t_{\infty}-q\varepsilon)\psi(\bar{n}))<\infty \nonumber
\end{eqnarray}
by definition of $t_{\infty}.$  Let $t=t_{\infty}-q\frac{\phi(\bar{n})}{-\psi(\bar{n})}+K$ such that $K \in \mathbb{R}.$
Then, if $K>0$ and $\varepsilon \rightarrow 0,$
$$\sum_{n=N}^{\infty}\exp((q\phi-t\psi)(\bar{n}))=\sum_{n=N}^{\infty}\exp(-(t_{\infty}+K)\psi(\bar{n}))<\infty$$
and if $K<0,$
$$\sum_{n=N}^{\infty}\exp((q\phi-t\psi)(\bar{n}))=\sum_{n=N}^{\infty}\exp(-(t_{\infty}+K)\psi(\bar{n}))=\infty$$
by definition of $t_{\infty}.$ \par
Thus, 
\[\inf\{t \in \mathbb{R}: \sum_{n=1}^{\infty}\exp((q\phi-t\psi)(\bar{n}))<\infty\}=t_{\infty}-q\alpha_{\lim}.\]

Therefore, for each $q \in \mathbb{R},$ 
$$\inf\{t \in \mathbb{R}:\mathcal{P}(q\phi-t\psi)<\infty\}=-\alpha_{\lim}q+t_{\infty}$$
by Equation~\ref{eq:4}.  By definition, it follows that
\[\tilde{t}(q)=-\alpha_{\lim}q+t_{\infty}.\]

It follows that $\tilde{t}(q)$ is a decreasing line because $\alpha_{\lim}>0.$
\end{proof}

The function $\tilde{t}(q)$ is connected to the following set $Q.$ 

\subsection{The Set $Q$ and The Function $\alpha(q)$}\hspace*{\fill} \par
Let $Q=\{q \in \mathbb{R}:T(q)=\tilde{t}(q)\}.$  Since $T(q)$ is strictly convex and $\tilde{t}(q)$ is linear, the following proposition is immediate.

\begin{proposition}\label{prop:Q1}
$Q$ can be a closed interval, half-open infinite interval, a point, or the empty set.
\end{proposition}

The following proposition immediately follows from Propositions \ref{prop:CD1}, \ref{prop:L1}, and \ref{prop:Q1}.
\begin{proposition}\label{prop:AAA1}
Let $\phi: \Sigma \rightarrow \mathbb{R}^{-}$ and $\psi: \Sigma \rightarrow \mathbb{R}^{+}$ be non-cohomologous locally H\"older potentials.  $T(q)$ has at most two phase transitions. 
\end{proposition}

We remark that Hanus, Mauldin, and Urba\'nski \cite{hanus2002thermodynamic} considered the families of potentials $f^{(i)},$ which is strongly H\"older, and $\log|\phi_{i}^{'}|$, which is defined by a regular, conformal iterated function system $\{\phi_{i}\}$ that satisfy the open set condition (both the terms regular conformal iterated function system and the open set condition are defined in Chapter 4 of \cite{mauldin2003graph}).  Their potentials give that $Q=\emptyset.$  Hence, the multifractal spectrum was analytic in their case.  An important function used in multifractal analysis is $\alpha(q),$ which is connected to $Q.$  By Proposition~\ref{prop:Q1}, $Q=[q_{0}, q_{1}]$ for some $q_{0},q_{1} \in \mathbb{R} \cup \{-\infty,\infty\}.$  

\begin{definition}
Let $Q=[q_{0},q_{1}]$ for some $q_{0},q_{1} \in \mathbb{R}.$  Then, we have a function $\alpha(q)$ such that
$$\alpha(q) =
\left\{
	\begin{array}{ll}
		-T'(q)  & \mbox{if } q \in Q^{\complement} \\
		 \alpha_{\lim}& \mbox{if } q \in (q_{0},q_{1}) \\
        \alpha^{-}=\lim\limits_{q \rightarrow q_{0}^{-}}\alpha(q)& \mbox{if } q=q_{0}>-\infty \\
         \alpha^{+}=\lim\limits_{q \rightarrow q_{1}^{+}}\alpha(q)& \mbox{if } q=q_{1}<\infty. \\
         \end{array}
\right.$$
If $q_{0}=-\infty,$ $\alpha^{-}=\alpha(q_{0})=\alpha_{\lim}$ and if $q_{1}=\infty,$ $\alpha^{+}=\alpha(q_{1})=\alpha_{\lim}.$
\end{definition}

If $Q$ is a singleton, then our preceding definition applies for $q_{0}=q_{1}<\infty.$  In that case, $\alpha^{-}=\alpha^{+}=\alpha(q_{0}).$  Our analysis of $\alpha(q)$ also depends on its extreme values.  We remark that the extreme values of $\alpha(q)$ are the supremum and infimum of the possible local dimension, which we will define later, of any sequence in $\Sigma.$

\begin{definition}
We let
$\alpha_{\inf}:=\inf\limits_{q \in \mathbb{R}}\alpha(q)\text{ and } \alpha_{\sup}:=\sup\limits_{q \in \mathbb{R}}\alpha(q).$
\end{definition}

We prove the following formula connecting $\alpha_{\inf}$ and $\alpha_{\sup}$ to a ratio involving $\phi$ and $\psi.$

\begin{lemma}
Let $\phi:\Sigma \rightarrow \mathbb{R}^{-}$ and $\psi:\Sigma \rightarrow \mathbb{R}^{+}$ be locally H\"older and $\Sigma$ have the BIP property.  Then,
\[\alpha_{\sup}=\sup\limits_{\nu \in M(\Sigma, \sigma)}\frac{\int \phi\, \mathrm{d}\nu}{-\int \psi\, \mathrm{d}\nu} \text{ and } \alpha_{\inf}=\inf\limits_{\nu \in M(\Sigma, \sigma)}\frac{\int \phi\, \mathrm{d}\nu}{-\int \psi\, \mathrm{d}\nu}.\]
\end{lemma}

\begin{proof}
By the variational principle,
\[\mathcal{P}(q\phi-T(q)\psi)=\sup_{\nu \in M(\sigma,\sigma)}\{q\int \phi\, \mathrm{d}\nu -T(q)\int \psi\, \mathrm{d}\nu + h(\nu)\} \le 0.\]

Because entropy is non-negative, we get that
\[\sup_{\nu \in M(\sigma,\sigma)}\{q\int \phi\, \mathrm{d}\nu -T(q)\int \psi\, \mathrm{d}\nu\} \le 0.\]

Then, for any $\nu \in M(\Sigma, \sigma),$
\[q\int \phi\, \mathrm{d}\nu-T(q)\int \psi\, \mathrm{d}\nu \le 0.\]

With a bit of rearrangement,
\[\frac{q\int \phi\, \mathrm{d}\nu}{\int \psi \mathrm{d}\nu} \le T(q).\]

Taking the derivatives of both sides with respect to $q,$ we get that
\[\frac{\int \phi\, \mathrm{d}\nu}{-\int \psi\, \mathrm{d}\nu} \ge -T'(q).\]

It follows that
\[\sup\limits_{\nu \in M(\Sigma, \sigma)}\frac{\int \phi\, \mathrm{d}\nu}{-\int \psi\, \mathrm{d}\nu} \ge -T'(q).\]

We remark that 
\[\sup\limits_{\nu \in M(\Sigma, \sigma)}\frac{\int \phi\, \mathrm{d}\nu}{-\int \psi\, \mathrm{d}\nu} \ge \alpha_{\lim}\] by definition of $\alpha_{\lim}$ and the approximation of locally H\"older potentials with locally constant functions.

Therefore, 
\[\sup\limits_{\nu \in M(\Sigma, \sigma)}\frac{\int \phi\, \mathrm{d}\nu}{-\int \psi\, \mathrm{d}\nu} = \alpha_{\sup}.\]

Using a similar argument, the result for the infimum of $\alpha(q)$ also follows:
\[\inf\limits_{\nu \in M(\Sigma, \sigma)}\frac{\int \phi\, \mathrm{d}\nu}{-\int \psi\, \mathrm{d}\nu} = \alpha_{\inf}.\]
\end{proof}

Note that $\alpha_{\lim} \in [\alpha_{\inf},\alpha_{\sup}]$ if it exists. Since $Q=[q_{0},q_{1}]$ for some $q_{0},q_{1} \in \mathbb{R},$ we immediately get the following decomposition from the definition of $\alpha(q):$
\[(\alpha_{\inf},\alpha_{\sup})=\{\alpha(q):q\in Q^{\complement}\}  \cup (\alpha^{+},\alpha_{\lim}) \cup \{\alpha_{\lim}\} \cup (\alpha_{\lim},\alpha^{-}) \cup \{\alpha^{-}\} \cup \{\alpha^{+}\}.\]

The function $\alpha(q)$ is always positive for every $q \in \mathbb{R}$ because Iommi \cite{iommi2005multifractal} (Pg 1892, Proposition 4.3) proved that $T(q)$ is a decreasing function of $q.$  We will soon notice that the multifractal spectrum is connected to the functions $T(q)$ and $\alpha(q).$

\subsection{The Multifractal Spectrum}\hspace*{\fill} \par
Before defining the multifractal spectrum we define symbolic dimension and the set $X_{\alpha}^{s}.$  We use the Gibbs measure $\mu$ for $\phi.$

\begin{definition}
A word $x=(x_{1},x_{2},...,x_{m},...) \in \Sigma$ has {\em symbolic dimension} $\alpha$ provided that
\[
d_{\mu}(x):=\lim\limits_{m \rightarrow \infty}\frac{\log\mu([x_{1},x_{2},...,x_{m}])}{\log|[x_{1},x_{2},...,x_{m}]|}=\alpha
\]
such that $\alpha_{\inf} \le \alpha \le \alpha_{\sup}.$
\end{definition}

We now consider sets with local dimension $\alpha$ as follows.

\begin{definition}
For each fixed $\alpha \in [\alpha_{\inf}, \alpha_{\sup}],$ we have the set
$$X_{\alpha}^{s}=\{x \in \Sigma: \lim\limits_{m \rightarrow \infty}\frac{\log\mu([x_{1},x_{2},...,x_{m}])}{\log|[x_{1},x_{2},...,x_{m}]|}=\alpha \}.$$
\end{definition}

Similarly, the local dimension of $x \in \Sigma$ is the limit
\[\lim\limits_{r \rightarrow 0}\frac{\log\mu(B(x,r))}{\log{r}}.\]  We now prove a proposition that the pointwise and symbolic dimensions are equal a.e. $x \in \Sigma.$

\begin{proposition}\label{prop:AAA}
Let $\phi:\Sigma \rightarrow \mathbb{R}^{-}$ be locally H\"older with Gibbs state $\mu$ and let $\mathcal{P}(\phi)<\infty.$
Take $\psi: \Sigma \rightarrow \mathbb{R}^{+}$ as a locally H\"older metric potential.  Assume that $\mu_{q}$ is the Gibbs state for $q\phi-T(q)\psi$ and $\phi,\psi \in \mathcal{L}^{1}(\mu_{q}).$  Then, for $\mu_{q}-$a.e. $x \in \Sigma,$  
\begin{equation}\label{eq:Symbolic}
\frac{\int \phi \mathrm{d}\mu_{q}}{-\int \psi \mathrm{d}\mu_{q}}= \lim\limits_{m \rightarrow \infty}\frac{\log\mu([x_{1},x_{2},...,x_{m}])}{\log|[x_{1},x_{2},...,x_{m}]|}=\alpha(q).
\end{equation}
Furthermore, the pointwise and local dimension are equal a.e. $\Sigma.$
\end{proposition}

\begin{proof}
Consider the set 
\[
\bar{X}:=\{x \in \Sigma: \lim\limits_{m \rightarrow \infty}\frac{\psi(\sigma^{m}(x))}{\sum_{n=0}^{m-1}\psi(\sigma^{m}(x))}=0\}=\{x \in \Sigma: \lim\limits_{m \rightarrow \infty}\frac{\sum_{n=0}^{m}\psi(\sigma^{n}(x))}{\sum_{n=0}^{m-1}\psi(\sigma^{n}(x))}=1\}.
\]
We will prove that $\mu(\bar{X})=1.$

First, we show that for any $x \in \bar{X},$
\[\lim\limits_{m \rightarrow \infty}\frac{\log\mu([x_{1},x_{2},...,x_{m}])}{\log|[x_{1},x_{2},...,x_{m}]|}=\lim\limits_{r \rightarrow 0}\frac{\log\mu(B(x,r))}{\log{r}}.\]

For each $x \in \bar{X},$ we can find $m \in \mathbb{N}$ large and $r>0$ such that the following construction holds.  If $|[x_{1},...,x_{m},x_{m+1}]| \le r \le |[x_{1},...,x_{m}]|,$ we find that $\mu(B(x,r))=\mu([x_{1},x_{2},...,x_{m+1}])$ for some $m \in \mathbb{N}.$  We get the following inequality:
\[\lim\limits_{m \rightarrow \infty}\frac{\log\mu([x_{1},..,x_{m},x_{m+1}])}{\log|[x_{1},...,x_{m}]|} \le \lim\limits_{r \rightarrow 0}\frac{\log\mu(B(x,r))}{\log{r}} \le \lim\limits_{m \rightarrow \infty}\frac{\log\mu([x_{1},....,x_{m},x_{m+1}])}{\log|[x_{1},...,x_{m},x_{m+1}]|}\]

Hence, we will prove that
\[\lim\limits_{m \rightarrow \infty}\frac{\log\mu([x_{1},..,x_{m},x_{m+1}])}{\log|[x_{1},...,x_{m}]|}=\lim\limits_{m \rightarrow \infty}\frac{\log\mu([x_{1},....,x_{m},x_{m+1}])}{\log|[x_{1},...,x_{m},x_{m+1}]|}.\]

These limits are equal if 
\[\lim\limits_{m \rightarrow \infty} \frac{\log|[x_{1},...,x_{m}]|}{\log|[x_{1},...,x_{m},x_{m+1}]|}=1.\]

For each $x \in \bar{X},$ there exists $C>0$ such that
\[\frac{1}{C} \le \frac{|[x_{1},...,x_{m}]|}{\prod_{n=0}^{m-1}(\exp(\psi(\sigma^{n}(x))))^{-1}} \le C \text{ and } \frac{1}{C} \le \frac{|[x_{1},...,x_{m},x_{m+1}]|}{\prod_{n=0}^{m}(\exp(\psi(\sigma^{n}(x))))^{-1}} \le C\]
because $\psi$ is a metric potential.

Then,
\[-\log{C} \le \log\left(\frac{|[x_{1},...,x_{m}]|}{\prod_{n=0}^{m-1}(\exp(\psi(\sigma^{n}(x))))^{-1}}\right) \le \log{C}\]
and
\[-\log{C} \le \log\left(\frac{|[x_{1},...,x_{m},x_{m+1}]|}{\prod_{n=0}^{m}(\exp(\psi(\sigma^{n}(x))))^{-1}}\right) \le \log{C}.\]

Hence, we have that
\[\log\left(\frac{|[x_{1},...,x_{m}]|}{\prod_{n=0}^{m-1}(\exp(\psi(\sigma^{n}(x))))^{-1}}\right)-\log\left(\frac{|[x_{1},...,x_{m+1}]|}{\prod_{n=0}^{m}(\exp(\psi(\sigma^{n}(x))))^{-1}}\right)=0.\]

This gives us that
\[\log\left(\frac{|[x_{1},...,x_{m}]|}{|[x_{1},...,x_{m},x_{m+1}]|}\right)-\log\left(\frac{\prod_{n=0}^{m}\exp(\psi(\sigma^{n}(x)))}{\prod_{n=0}^{m-1}\exp(\psi(\sigma^{n}(x)))}\right)=0,\]
so
\[\lim\limits_{m \rightarrow \infty}\log\left(\frac{|[x_{1},...,x_{m}]|}{|[x_{1},...,x_{m},x_{m+1}]|}\right)-\log\left(\frac{\prod_{n=0}^{m}\exp(\psi(\sigma^{n}(x)))}{\prod_{n=0}^{m-1}\exp(\psi(\sigma^{n}(x)))}\right)=0.\]

We find that
\[\lim\limits_{m \rightarrow \infty}\log\left(\frac{\prod_{n=0}^{m}\exp(\psi(\sigma^{n}(x)))}{\prod_{n=0}^{m-1}\exp(\psi(\sigma^{n}(x)))}\right)=\lim\limits_{m \rightarrow \infty}\sum_{n=0}^{m}\exp(\psi(\sigma^{n}(x)))-\sum_{n=0}^{m-1}\exp(\psi(\sigma^{n}(x)))=0\]
because $x \in \bar{X}.$

It follows that
\[\lim\limits_{m \rightarrow \infty} \frac{\log|[x_{1},...,x_{m}]|}{\log|[x_{1},...,x_{m},x_{m+1}]|}=1.\]

Hence, we find that
\[\lim\limits_{m \rightarrow \infty}\frac{\log\mu([x_{1},..,x_{m},x_{m+1}])}{\log|[x_{1},...,x_{m}]|} = \lim\limits_{r \rightarrow 0}\frac{\log\mu(B(x,r))}{\log{r}} = \lim\limits_{m \rightarrow \infty}\frac{\log\mu([x_{1},....,x_{m},x_{m+1}])}{\log|[x_{1},...,x_{m},x_{m+1}]|}.\]

Therefore, the pointwise and local dimension for each $x \in \bar{X}$ are equal:
\[\lim\limits_{r \rightarrow 0}\frac{\log\mu(B(x,r))}{\log{r}} = \lim\limits_{m \rightarrow \infty}\frac{\log\mu([x_{1},....,x_{m}])}{\log|[x_{1},...,x_{m}]|}.\]

Additionally, we find that $x \in \bar{X}$ satisfy
\[\sum_{k=m}^{\infty}\psi(\sigma^{k}(x))<\infty, \text{ so } \sum_{k=1}^{\infty}\psi(\sigma^{k}(x))<\infty\]
for large $m \in \mathbb{N}.$

Since $\mu$ is Gibbs for $\phi,$ $\psi$ is a metric potential, and the local and pointwise dimension for $x \in \bar{X}$ are equal, we immediately find that
\[\lim\limits_{r \rightarrow 0}\frac{\log\mu(B(x,r))}{\log{r}}=\lim\limits_{m \rightarrow \infty}\frac{\log\mu([x_{1},....,x_{m}])}{\log|[x_{1},...,x_{m}]|}=\lim\limits_{m \rightarrow \infty}\frac{\sum_{j=0}^{m-1}\phi(\sigma^{j}(x))}{-\sum_{j=0}^{m-1}\psi(\sigma^{j}(x))}=\frac{\int \phi \mathrm{d}\mu_{q}}{-\int \psi \mathrm{d}\mu_{q}}\]
for each $x \in \bar{X}$ by the Birkhoff ergodic theorem.  Note that Iommi \cite{iommi2005multifractal} proved this result for a.e. $x \in \Sigma.$

It immediately follows that $\mu(\bar{X})=1.$  Hence, the pointwise and local dimension are equal a.e.
\end{proof}

We give the following result as an alternate characterisation of $\alpha_{\inf}$ and $\alpha_{\sup}$ because Proposition \ref{prop:AAA} can be applied to all $q \in \mathbb{R}.$

\begin{lemma}
Let $\phi:\Sigma \rightarrow \mathbb{R}^{-}$ be locally H\"older with Gibbs state $\mu$ and let $\mathcal{P}(\phi)<\infty.$
Take $\psi: \Sigma \rightarrow \mathbb{R}^{+}$ as a locally H\"older metric potential.  Assume that $\mu_{q}$ is the Gibbs state for $q\phi-T(q)\psi$ and $\phi,\psi \in \mathcal{L}^{1}(\mu_{q}).$  Then, 
\[\alpha_{\inf}=\inf\{d_{\mu}(x): x \in \Sigma\} \text{ and } \alpha_{\sup}=\sup\{d_{\mu}(x): x \in \Sigma\}.\]
\end{lemma}

\begin{proof}
The result follows from Proposition \ref{prop:AAA} and the definition of $\alpha_{\inf}$ and $\alpha_{\sup}.$
\end{proof}

As Proposition \ref{prop:AAA} infers, we can consider the set
\[X_{\alpha}:=\{x \in \Sigma: \lim\limits_{r \rightarrow 0}\frac{\log\mu(B(x,r))}{\log{r}}=\alpha\}.\]  Since $X_{\alpha}=X_{\alpha}^{s}$ a.e., we use $X_{\alpha}^{s}$ in our definition of the multifractal spectrum.

\begin{definition}
For each $\alpha \in (\alpha_{\inf}, \alpha_{\sup}),$ the {\em multifractal spectrum} is the function $f_{\mu}(\alpha)$ defined by
\begin{equation}\label{eq:E}
\alpha \mapsto \dim_{H}(X_{\alpha}^{s}).
\end{equation}
\end{definition}

We remark that the multifractal spectrum depends on the measure $\mu.$  Iommi's work \cite{iommi2005multifractal} has a theorem stating that the multifractal spectrum is a Legendre transform.  Hence, we will define the concepts of Fenchel and Legendre transforms.

\begin{definition}
Let $h$ be a convex function.  $(h,g)$ is called a {\em Fenchel pair} if
\[g(p)=\sup\limits_{x}\{px-h(x)\}.\]
Alternatively, we say that $g$ is the {\em Fenchel transform} of $h.$  If $h$ is a convex, twice-differentiable function, then $g$ is called a {\em Legendre transform}.
\end{definition}

We provide Iommi's theorem from Pg 11, Theorem 4.1 of  \cite{iommi2005multifractal}, which proves that the multifractal spectrum is a Legendre transform.  

\begin{theorem}
Let $\phi:\Sigma \rightarrow \mathbb{R}^{-}$ be locally H\"older with Gibbs state $\mu$ and let $\mathcal{P}(\phi)<\infty.$
Take $\psi: \Sigma \rightarrow \mathbb{R}^{+}$ as a locally H\"older metric potential.  The multifractal spectrum $f_{\mu}$ is the Fenchel transform of $T.$
\end{theorem}

We provide the following remark.  Iommi proves that $T(q)$ is a convex function.  If we take $x=q,$ $h(q)=T(q),$ $h'(q)=-\alpha(q),$ and $p=-\alpha,$
\[f_{\mu}(\alpha)=g(p)=\sup\limits_{x \in \mathbb{R}}\{px-h(x)\}=\sup\limits_{q \in \mathbb{R}}\{-q\alpha-T(q)\}=\inf\limits_{q\in\mathbb{R}}\{T(q)+q\alpha\}.\]
Hence, $(T,f_{\mu})$ form a Fenchel pair and $f_{\mu}(\alpha)$ is a Lengendre transform.  Thus, we will use the following form of Iommi's theorem to prove Theorem \ref{thm:M1}.

\begin{theorem}\label{thm:I1}
Let $\phi:\Sigma \rightarrow \mathbb{R}^{-}$ be locally H\"older with Gibbs state $\mu$ and let $\mathcal{P}(\phi)<\infty.$
Take $\psi: \Sigma \rightarrow \mathbb{R}^{+}$ as a locally H\"older metric potential.  For each $\alpha \in (\alpha_{\inf}, \alpha_{\sup}),$
\[f_{\mu}(\alpha)=\inf\limits_{q \in \mathbb{R}}\{T(q)+q\alpha\}=\dim_{H}(X_{\alpha}^{s}).\]
\end{theorem}

The following proposition is immediate.
\begin{proposition}\label{prop:AA1}
Let $\phi:\Sigma \rightarrow \mathbb{R}^{-}$ be locally H\"older with Gibbs state $\mu$ and let $\mathcal{P}(\phi)<\infty.$
Take $\psi: \Sigma \rightarrow \mathbb{R}^{+}$ as a locally H\"older metric potential.  Furthermore, assume that $\phi$ and $\psi$ are non-cohomologous.  If $T(q)$ is analytic over $\mathbb{R},$ then $f_{\mu}(\alpha(q))$ is analytic over $(\alpha_{\inf}, \alpha_{\sup}).$
\end{proposition}

Our paper is on the multifractal spectrum's phase transitions; however, we gave the preceding proposition for completeness.  In order to analyse the phase transitions of the multifractal spectrum, we must analyse the phase transitions of $T(q).$  To do this, we must give criteria for the existence of special measures for our potential.

\subsection{Measures for Our Potentials}\hspace*{\fill} \par
In this section, we provide criteria for the existence and uniqueness of special measures for $\phi$ and $q\phi-T(q)\psi.$  We have the following (modified) result by Sarig \cite{sarig2003existence} (Pg 2, Theorem 1) and Mauldin and Urba\'nski \cite{mauldin2003graph} (Pg 14, Theorem 2.2.4).
\begin{theorem}\label{thm:S2}
Let $(\Sigma,\sigma)$ be topologically mixing and $\gamma: \Sigma \rightarrow \mathbb{R}$ is locally H\"older.  Then, $\gamma$ has a unique invariant Gibbs state if and only if the transition matrix $A$ has the BIP property and $\mathcal{P}(\gamma)<\infty.$ 
\end{theorem}

Thus, by Sarig, Mauldin, and Urba\'nski, $\phi$ has a corresponding unique Gibbs measure $\mu.$  We now consider the family of potentials $q\phi-T(q)\psi.$

\begin{theorem}\label{thm:MR1}
Assume that $\Sigma$ satisfies the BIP property.  Let $\phi:\Sigma \rightarrow \mathbb{R}^{-}$ be locally H\"older with Gibbs state $\mu$ and let $\mathcal{P}(\phi)<\infty.$
Take $\psi: \Sigma \rightarrow \mathbb{R}^{+}$ as a locally H\"older metric potential.  For each $q \in Q^{\complement},$ there exists a unique, ergodic Gibbs state $\mu_{q}$ for $q\phi-T(q)\psi.$  Furthermore, $\mu_{q}(X_{\alpha(q)}^{s})=1.$  If we also have that $\int q\phi-T(q)\psi\, \mathrm{d}\mu_{q}>-\infty,$ $\mu_{q}$ is the unique equilibrium state for $q\phi-T(q)\psi.$
\end{theorem}
\begin{proof}
We have that $\Sigma$ is topologically mixing and BIP, the transition matrix $A$ satisfies the BIP property, and $q\phi-T(q)\psi$ is assumed to be locally H\"older.  Hence, a unique Gibbs state for $q\phi-T(q)\psi$ exists by Theorem \ref{thm:S2}.  Since $\Sigma$ is BIP, $\phi$ and $\psi$ are locally H\"older, and $q\phi-T(q)\psi$ has a Gibbs state, we have the unique, invariant, and ergodic Gibbs state $\mu_{q}$ for $q\phi-T(q)\psi$ by Theorem \ref{thm:S2}.  It follows that $\mu_{q}(X_{\alpha(q)}^{s})=1$ because $\mu_{q}$ can be normalised on $X_{\alpha(q)}^{s}.$  Thus, if $-q\phi+T(q)\psi$ is integrable, we have that $\mu_{q}$ is the unique equilibrium state for $q\phi-T(q)\psi.$
\end{proof}

An important assumption for our potentials $-\phi,\psi$ is that they are in $\mathcal{L}^{1}(\mu_{q})$ for each $q \in Q^{\complement}.$  Before we provide a corollary complementing Theorem \ref{thm:MR1}, we give a definition from Pg 6 of Sarig \cite{sarig2003existence}.  This definition will help us form a corollary about the function $\alpha(q)$ with respect to $\mu_{q}.$

\begin{definition}
Let $\phi:\Sigma \rightarrow \mathbb{R}^{-}$ be a locally H\"older potential such that $\mathcal{P}(\phi)<\infty.$  Denote $Dir(\phi)$ as the collection of all $\psi: \Sigma \rightarrow \mathbb{R}^{+}$ such that there exists $C_{\psi}>0, r \in (0,1)$ and $\varepsilon>0$ such that
\begin{enumerate}
\item $V_{n}(\psi)<C_{\psi}r^{n} \text{ for all } n \ge 1$ 
\item $\mathcal{P}(q\phi-t\psi)<\infty \text{ for all } t \in (T(q)-\varepsilon, T(q)+\varepsilon).$
\end{enumerate}
\end{definition}

Since $\psi$ is locally H\"older and there exist $\tilde{t}(q)<t<T(q)$ such that $\mathcal{P}(q\phi-t\psi)<\infty$ for $q \in Q^{\complement},$ it follows that $\psi \in Dir(\phi).$  Now, we provide a proposition (which is a modified version of Corollary 4 on Pg 6 of Sarig \cite{sarig2003existence}). 

\begin{proposition}\label{cor:AN1}
If $\Sigma$ satisfies the BIP property, $\mathcal{P}(\phi)<\infty,$ and there exist $B>0$ and $\theta \in (0,1)$ such that $V_{n}(\phi)<B\theta^{n}$ for all $n \ge 1,$ then for every $\psi \in Dir(\phi)$ and fixed $q \in Q^{\complement},$ there exists $\varepsilon_{0}>0$ for which $t \mapsto \mathcal{P}(q\phi-t\psi)$ is real analytic on $(T(q)-\varepsilon_{0}, T(q)+\varepsilon_{0}).$
\end{proposition}

Hence, we have that $t \mapsto \mathcal{P}(q\phi-t\psi)$ is real analytic on $(T(q)-\varepsilon_{0}, T(q)+\varepsilon_{0}).$  We will use this important fact in the next section.  Note that we can also apply Sarig's results to prove that $\mu$ is a Gibbs measure for $\phi.$  Furthermore, by Theorem \ref{thm:S2}, $\mu$ is the unique, ergodic Gibbs state for $\phi.$  Since we now have definitions and necessary results from thermodynamic formalism, we can now prove Theorem \ref{thm:M1} and Theorem \ref{thm:M2}, which are about the phase transitions of the multifractal spectrum.

\section{Proof of Theorems 1.1 and 1.2}
We will prove that the multifractal spectrum is analytic by taking advantage of a decomposition for $(\alpha_{\inf},\alpha_{\sup}).$  We take these steps in our proof.
\begin{enumerate}[I]
\item  Let $Q=[q_{0},q_{1}]$ for some $q_{0},q_{1} \in \mathbb{R} \cup \{-\infty,\infty\}.$  We prove that the functions $t \mapsto \mathcal{P}(q\phi-t\psi),$ $T(q),$ and $\alpha(q)$ are analytic on open subintervals of $Q^{\complement}.$  Then, we prove that the multifractal spectrum is analytic on open subintervals of $\{\alpha(q): q \in Q^{\complement}\}.$  
\item  We prove that the multifractal spectrum is analytic on $(\alpha^{+},\alpha_{\lim})$ and $(\alpha_{\lim},\alpha^{-}).$
\item  Finally, we assume that $Q=[q_{0},q_{1}]$ such that $0<q_{0}<q_{1}<\infty$ and $0< \alpha_{\lim} \le \infty$ exists.  Our result about multifractal spectrum's phase transitions follows. 
\end{enumerate}

\subsection{The Set $Q^{\complement}$ and Its Connection to $T(q)$}\hspace*{\fill} \par
First, we recall the definition of $Q:$
\[Q:=\{q \in \mathbb{R}: T(q)=\tilde{t}(q)\}.\]
We will prove that the functions $f_{\mu}(\alpha),$ $\alpha(q),$ and $T(q)$ are analytic on sets related to $Q^{\complement}.$

Fix $q \in Q^{\complement}.$  By Proposition \ref{cor:AN1}, $\mathcal{P}(q\phi-t\psi)$ is an analytic function of $t$ in a neighbourhood around $T(q).$  Hence, we can now prove that $\alpha(q)$ is analytic on open subintervals of $Q^{\complement}.$

\begin{proposition}\label{prop:B1}
Let $\phi:\Sigma \rightarrow \mathbb{R}^{-}$ be locally H\"older with Gibbs state $\mu$ and let $\mathcal{P}(\phi)<\infty.$
Take $\psi: \Sigma \rightarrow \mathbb{R}^{+}$ as a locally H\"older metric potential.  Furthermore, assume that $\phi$ and $\psi$ are non-cohomologous.  $T(q)$ is strictly convex, well defined, and analytic on open subintervals $Q^{\complement}.$  Furthermore, $\alpha(q)$ is well defined and analytic on open subintervals of $Q^{\complement}.$  
\end{proposition}

\begin{proof}
Fix an arbitrary $q \in Q^{\complement}.$  By Theorem \ref{thm:MR1}, $\mu_{q}$ is Gibbs for $q\phi-T(q)\psi.$  Let $\varepsilon>0.$  Then, $t \mapsto \mathcal{P}(q\phi-t\psi)$ is an analytic function of $t$ in an $\varepsilon$-neighbourhood around $T(q)$ by Corollary \ref{cor:AN1}.  Consider $\mathcal{P}(q\phi-t\psi)$ for $t \in (T(q)-\varepsilon, T(q)+\varepsilon).$
Since $-\phi,\psi \in \mathcal{L}^{1}(\mu_{q}),$ $\int -(q\phi-T(q)\psi)\, \mathrm{d}\mu_{q}<\infty.$  Hence, $\mu_{q}$ is the equilibrium state for $q\phi-T(q)\psi.$  We will denote $\mathcal{P}(q\phi-t\psi)$ as $\mathcal{P}(q,t).$  Then, by Proposition 2.6.13 on Pg 47 of Mauldin and Urba\'nski \cite{mauldin2003graph}, we can take the derivative of $\mathcal{P}(q\phi-t\psi)$:
\begin{equation}\label{eq:F}
\frac{\partial P(q,t)}{\partial t}=-\int_{\Sigma} \psi\, \mathrm{d}\mu_{q}.
\end{equation}

Since $\psi(x)>0$ for every $x \in \Sigma,$ Equation (\ref{eq:F}) gives us that
$$\frac{\partial P(q,t)}{\partial t}<0 \text{ and } \mathcal{P}(q\phi-T(q)\psi)=0.$$

Hence, $T(q)$ is well defined and analytic by the implicit function theorem.  Since $T(q)$ is strictly decreasing and $\phi$ and $\psi$ are non-cohomologous to each other, $T(q)$ is strictly convex.  Furthermore, $\alpha(q)$ is well defined and analytic because $T(q)$ is analytic and strictly convex.  
\end{proof}

We now prove some results about the multifractal spectrum.

\begin{lemma}
Let $\phi:\Sigma \rightarrow \mathbb{R}^{-}$ be locally H\"older with Gibbs state $\mu$ and let $\mathcal{P}(\phi)<\infty.$
Take $\psi: \Sigma \rightarrow \mathbb{R}^{+}$ as a locally H\"older metric potential.  Furthermore, assume that $\phi$ and $\psi$ are non-cohomologous.  For each $\alpha \in \{\alpha(q): q \in Q^{\complement}\},$
\begin{equation}\label{eq:G}
f_{\mu}(\alpha)=f_{\mu}(\alpha(q))=T(q)+q\alpha(q).
\end{equation}
\end{lemma}
\begin{proof}
Take $\alpha=\alpha(q)$ for some $q \in Q^{\complement}.$  By Theorem \ref{thm:I1},
$$f_{\mu}(\alpha)=\inf\limits_{q \in \mathbb{R}}\{T(q)+q\alpha\}.$$
Then, since $T(q)$ is analytic in a neighbourhood of our $q \in Q^{\complement},$
$\frac{\mathrm{d}}{\mathrm{d}q}(T(q)+q\alpha)=0$ when $T'(q)+\alpha=0.$  We exactly have that $\alpha=-T'(q)=\alpha(q).$  Hence, 
$$f_{\mu}(\alpha)=T(q)+q\alpha(q)$$
for our $q \in Q^{\complement}.$
\end{proof}

This lemma gives us a formula we need for the proof of the following proposition.

\begin{proposition}\label{prop:A1}
Let $\phi:\Sigma \rightarrow \mathbb{R}^{-}$ be locally H\"older with Gibbs state $\mu$ and let $\mathcal{P}(\phi)<\infty.$
Take $\psi: \Sigma \rightarrow \mathbb{R}^{+}$ as a locally H\"older metric potential.  Furthermore, assume that $\phi$ and $\psi$ are non-cohomologous. The multifractal spectrum is analytic and strictly concave on any open subinterval $S$ of $\{\alpha(q): q \in Q^{\complement}\}.$
\end{proposition}

\begin{proof}
Let $\alpha=\alpha(q)$ s.t. $q \in Q^{\complement}.$  Remember that $\alpha(q)$ is well defined and analytic on $Q^{\complement}.$  By Equation (\ref{eq:G}), $f_{\mu}(\alpha(q))=T(q)+q\alpha(q).$  To prove that $f_{\mu}(\alpha)$ is analytic as a function of $\alpha,$ we will invert $\alpha(q).$  We now take the derivative of $f_{\mu}(\alpha(q)),$ which is also used in the proof of Lemma 6.17 in Pg 89 of Barreira \cite{barreira2008dimension}:
\begin{equation}\label{eq:H}
\frac{\mathrm{d}}{\mathrm{d}q}f_{\mu}(\alpha(q))= \frac{d}{d\alpha(q)}(f_{\mu}(\alpha(q)))\alpha'(q)=q\alpha'(q).
\end{equation}

Then, the derivative with respect to $\alpha(q)$ of the multifractal spectrum is 
\[\frac{\mathrm{d}}{\mathrm{d}\alpha(q)}(f_{\mu}(\alpha(q)))=q.\] 
Because we took the derivative in terms of $\alpha(q),$ $q$ is a function of $\alpha,$ i.e., $q=q(\alpha).$  Since $T(q)$ is strictly convex on $Q^{\complement},$ $-T''(q)<0$ for each $q \in Q^{\complement}.$  Because $\alpha'(q)=-T''(q)<0,$ $\alpha(q)$ and $q(\alpha)$ are invertible.  Hence, since $\alpha(q)$ is analytic, $\frac{\mathrm{d}}{\mathrm{d}\alpha(q)}(f_{\mu}(\alpha(q)))=q(\alpha)$ is analytic.  Thus, since $T(q(\alpha))$ and $q(\alpha)$ are analytic, $f_{\mu}(\alpha)$ is analytic as a function of $\alpha.$

To prove the strict concavity of $f_{\mu}(\alpha),$ we take further derivatives of the multifractal spectrum with respect to $q.$  Then, it follows that
$$1=\frac{\mathrm{d}^2}{\mathrm{d}q^2}(f_{\mu}(\alpha(q)))=\frac{\mathrm{d}^{2}}{\mathrm{d}\alpha^{2}}(f_{\mu}(\alpha(q)))\alpha'(q).$$

We have proven that
$$\frac{\mathrm{d}^2}{\mathrm{d}\alpha^2}(f_{\mu}(\alpha(q)))=\frac{1}{\alpha'(q)}=\frac{-1}{T''(q)}<0$$
because $\phi$ and $\psi$ are not cohomologous to each other (hence, $T(q)$ is strictly convex).  Thus, the multifractal spectrum is strictly concave on any open subinterval $S \subset \{\alpha(q): q \in Q^{\complement}\}.$
\end{proof}

The proof of Proposition \ref{prop:A1} also establishes two further results about $\alpha(q)$ and $f_{\mu}(\alpha)$ respectively.

\begin{lemma}
Let $\phi:\Sigma \rightarrow \mathbb{R}^{-}$ be locally H\"older with Gibbs state $\mu$ and let $\mathcal{P}(\phi)<\infty.$
Take $\psi: \Sigma \rightarrow \mathbb{R}^{+}$ as a locally H\"older metric potential.  Furthermore, assume that $\phi$ and $\psi$ are non-cohomologous.  $\alpha(q)$ is a strictly decreasing function on open subintervals of $Q^{\complement}.$
\end{lemma}
\begin{proof}
Since $\alpha'(q)<0$ on $Q^{\complement},$ the lemma follows from Proposition \ref{prop:A1}.
\end{proof}

\begin{proposition}\label{prop:ID1}
Let $\phi:\Sigma \rightarrow \mathbb{R}^{-}$ be locally H\"older with Gibbs state $\mu$ and let $\mathcal{P}(\phi)<\infty.$
Take $\psi: \Sigma \rightarrow \mathbb{R}^{+}$ as a locally H\"older metric potential.  Furthermore, assume that $\phi$ and $\psi$ are non-cohomologous.  The multifractal spectrum $f_{\mu}(\alpha)$
\begin{enumerate}
\item increases on open subintervals $S$ of $\{\alpha(q): q<0\} \cap \{\alpha(q): q \in Q^{\complement}\}$ 
\item decreases on open subintervals of $\{\alpha(q): q>0\} \cap \{\alpha(q): q \in Q^{\complement}\}.$
\end{enumerate}
\end{proposition}
\begin{proof}
Equation (\ref{eq:H}) gives us that 
$$\frac{\mathrm{d}}{\mathrm{d}\alpha}(f_{\mu}(\alpha(q))=q\alpha'(q)=-qT''(q).$$

Since $-T''(q)<0$ on $Q^{\complement},$ we get the following.
$\text{If } q<0, \frac{\mathrm{d}}{\mathrm{d}\alpha}(f_{\mu}(\alpha(q))>0.$ $\text{If } q>0, \frac{\mathrm{d}}{\mathrm{d}\alpha}(f_{\mu}(\alpha(q))<0.$  Hence, the increasing and decreasing behaviour on open subintervals of $\{\alpha(q):q\in Q^{\complement}\}$ is immediate.
\end{proof}

In summary, we proved that if $T(q)$ is analytic on open subintervals of $Q^{\complement},$ then $\alpha(q)$ is analytic on open subintervals of $Q^{\complement}.$  In turn, this gives us that $f_{\mu}(\alpha)$ is analytic on open subintervals of $\{\alpha(q):q\in Q^{\complement}\}.$  We also find that increasing and decreasing behaviour on open subintervals of $\{\alpha(q):q\in Q^{\complement}\}$ is based on the sign of each $q\in Q^{\complement}.$

\subsection{The Intervals $(\alpha^{+},\alpha_{\lim})$ and $(\alpha_{\lim},\alpha^{-})$}\hspace*{\fill} \par
Now that we have proven that the multifractal spectrum is analytic on open subintervals of $\{\alpha(q):q\in Q^{\complement}\},$ we must consider the behaviour of the multifractal spectrum on other open subintervals of $(\alpha_{\inf},\alpha_{\sup}).$  Without loss of generality, we assume that $Q=[q_{0},q_{1}]$ for $-\infty<0<q_{0}<q_{1}<\infty.$  As stated earlier, $\alpha^{-}=\alpha(q_{0})$ and $\alpha^{+}=\alpha(q_{1}).$   Using our decomposition of $(\alpha_{\inf},\alpha_{\sup}),$ we know that we must prove that the multifractal spectrum is analytic on $(\alpha^{+},\alpha_{\lim})$ and $(\alpha_{\lim},\alpha^{-}).$  

\begin{proposition}\label{prop:A2}
Let $\phi:\Sigma \rightarrow \mathbb{R}^{-}$ be locally H\"older with Gibbs state $\mu$ and let $\mathcal{P}(\phi)<\infty.$
Take $\psi: \Sigma \rightarrow \mathbb{R}^{+}$ as a locally H\"older metric potential.  Furthermore, assume that $\phi$ and $\psi$ are non-cohomologous. 
\begin{enumerate}
\item The function $f_{\mu}(\alpha)=T(q_{1})+q_{1}\alpha$ on $(\alpha^{+},\alpha_{\lim})$ and $f_{\mu}(\alpha)=T(q_{0})+q_{0}\alpha$ on $(\alpha_{\lim},\alpha^{-}).$  
\item The multifractal spectrum is an increasing linear function on $(\alpha^{+},\alpha_{\lim})$ and $(\alpha_{\lim},\alpha^{-})$ because $q_{0},q_{1}>0.$  
\item Furthermore, the sign of $q_{0}$ and $q_{1}$ determine the increasing or decreasing behaviour of the multifractal spectrum on $(\alpha^{+},\alpha_{\lim})$ and $(\alpha_{\lim},\alpha^{-}).$ 
\end{enumerate}
\end{proposition}
\begin{proof}
Let $\alpha \in (\alpha(q_{1}), \alpha_{\lim}).$  For each $\alpha \in (\alpha(q_{1}), \alpha_{\lim}),$ we have that 
$$f_{\mu}(\alpha)= \inf\limits_{q \in \mathbb{R}}\{T(q)+q\alpha\}=T(q_{1})+q_{1}\alpha.$$

Hence, $$\frac{\mathrm{d}}{\mathrm{d}\alpha}f_{\mu}(\alpha)=q_{1}>0.$$

Thus, $f_{\mu}(\alpha) $ is an increasing linear function with slope $q_{1}$ on the interval $(\alpha^{+}, \alpha_{\lim}).$  Assume that $\alpha \in (\alpha_{\lim}, \alpha(q_{0})).$  For each $\alpha \in (\alpha_{\lim}, \alpha(q_{0})),$ 
$f_{\mu}(\alpha)= \inf \limits_{q \in \mathbb{R}}\{T(q)+q\alpha\}=T(q_{0})+q_{0}\alpha.$  Hence, $$\frac{\mathrm{d}}{\mathrm{d}\alpha}f_{\mu}(\alpha)=q_{0}>0.$$

Thus, $f_{\mu}(\alpha)$ is an increasing linear function with slope $q_{0}$ on the interval $(\alpha_{\lim}, \alpha^{-}).$
\end{proof}

The multifractal spectrum is linear on $(\alpha_{\lim},\alpha^{-})$ and $(\alpha^{+},\alpha_{\lim})$ because of the endpoints of $Q.$  We can finally prove our main theorems.

\subsection{Phase Transitions- Intervals and Points of $Q$}\hspace*{\fill} \par
Now that we have proven results relating the analyticity of $T(q)$ to the analyticity of $f_{\mu}(\alpha),$ we use Propositions \ref{prop:Q1}, \ref{prop:A1}, and \ref{prop:A2} to prove that the multifractal spectrum has $0$ to $3$ phase transitions when $\alpha_{\lim}$ exists and $0<\alpha_{\lim}<\infty.$  Furthermore, we prove a complementary result when $\alpha_{\lim}=\infty.$  We provide the subsequent proposition to remind the reader about the forms $Q$ can take.

\begin{proposition}
$Q$ can be a closed interval, a half-open infinite interval, a point, or the empty set.
\end{proposition}

Analysing $Q$ gives us information about the number of phase transitions for the multifractal spectrum.  We give a thorough discussion of the multifractal spectrum's phase transitions in the case that $Q$ is a closed interval with positive endpoints.

\subsubsection{Positive Closed Interval}\hspace*{\fill} \par
We assume that $Q=[q_{0}, q_{1}]$ for some $q_{0}, q_{1} \in \mathbb{R}$ such that $0<q_{0}<q_{1}<\infty.$    

\begin{proposition}\label{prop:PCD1}
Let $\phi:\Sigma \rightarrow \mathbb{R}^{-}$ be locally H\"older with Gibbs state $\mu$ and let $\mathcal{P}(\phi)<\infty.$
Take $\psi: \Sigma \rightarrow \mathbb{R}^{+}$ as a locally H\"older metric potential.  Furthermore, assume that $\phi$ and $\psi$ are non-cohomologous. We have four possible types of behaviour for the multifractal spectrum as follows.  We call them cases 1 to 4 with respect to the order below.
\begin{enumerate}
\item Let $\alpha(q_{0})>\alpha_{\lim}>\alpha(q_{1}).$  The function $f_{\mu}(\alpha)$ is analytic on $(\alpha_{\inf}, \alpha(q_{1})),$  $(\alpha(q_{1}), \alpha_{\lim}),$ $(\alpha_{\lim}, \alpha(q_{0})),$ and $(\alpha(q_{0}), \alpha_{\sup}).$  There are three phase transitions for the multifractal spectrum at $\alpha(q_{0}), \alpha_{\lim},$ and $\alpha(q_{1}).$
\item Let $\alpha(q_{0})> \alpha_{\lim}=\alpha(q_{1}).$  The function $f_{\mu}(\alpha)$ is analytic on $(\alpha_{\inf}, \alpha_{\lim}),$ $(\alpha_{\lim}, \alpha(q_{0})),$ $(\alpha(q_{0}), \alpha_{\sup}).$  The multifractal spectrum has two phase transitions at $\alpha(q_{0})$ and $\alpha_{\lim}.$
\item Let $\alpha(q_{0})=\alpha_{\lim}>\alpha(q_{1}).$  The function $f_{\mu}(\alpha)$ is analytic on $(\alpha_{\inf}, \alpha(q_{1})),$ $(\alpha(q_{1}), \alpha_{\lim}),$ $(\alpha_{\lim},\alpha_{\sup}).$  The multifractal spectrum two phase transitions at $\alpha(q_{1})$ and $\alpha_{\lim}.$
\item Let $\alpha(q_{0})=\alpha_{\lim}=\alpha(q_{1}).$  The function $f_{\mu}(\alpha)$ is analytic on $(\alpha_{\inf}, \alpha_{\lim}),$ $(\alpha_{\lim},\alpha_{\sup}).$  The multifractal spectrum has a possible phase transition at $\alpha_{\lim}.$  
\end{enumerate}
\end{proposition}

\begin{proof}
The method to prove cases 2 to 4 is similar the proof for case 1, so we will only provide the proof for case 1.  Remember that $\alpha(q_{1})=\alpha^{+}$ and $\alpha(q_{0})=\alpha^{-}.$  We use the decomposition of $[\alpha_{\inf},\alpha_{\sup}]$ in the following way.  Each $\alpha \in (\alpha_{\inf}, \alpha(q_{1}))$ satisfies $\alpha=\alpha(q)$ for $q > q_{1}.$  Proposition \ref{prop:A2} gives us that $f_{\mu}(\alpha)$ equals $T(q_{1})+q_{1}\alpha$ on $(\alpha^{+},\alpha_{\lim})$ and $T(q_{0})+q_{0}\alpha$ on $(\alpha_{\lim},\alpha^{-}).$  Each $\alpha \in (\alpha(q_{0}), \alpha(0))$ satisfies $\alpha=\alpha(q)$ for a unique $0< q \le q_{0}.$  Each $\alpha \in (\alpha(0), \alpha_{\sup})$ satisfies $\alpha=\alpha(q)$ for a unique $q<0.$

Hence, by Propositions \ref{prop:A1} and \ref{prop:A2}, the multifractal spectrum is analytic on $(\alpha_{\inf},\alpha^{+}),$ $(\alpha^{+},\alpha_{\lim}),$ $(\alpha_{\lim},\alpha^{-}),$ and $(\alpha^{-},\alpha_{\sup}).$  The increasing and decreasing behaviour of the multifractal spectrum is immediate from Proposition \ref{prop:ID1} and Lemma \ref{prop:A2}.  Thus, it follows that the multifractal spectrum has phase transitions at $\alpha(q_{1}),$ $\alpha_{\lim},$ and $\alpha(q_{0}).$ 
\end{proof}

The proofs for the following other cases for $Q=[q_{0}, q_{1}]$ are proved in the same way as Proposition \ref{prop:PCD1}.  In all of the following cases for $Q,$ the multifractal spectrum ranges from having no phase transitions if $\alpha^{-}=\alpha_{\lim}=\alpha^{+}$ to three phase transitions if $\alpha^{-}>\alpha_{\lim}>\alpha^{+}.$  Using the same techniques as the proof to Proposition \ref{prop:PCD1}, we get the behaviour as outlined below. 
\begin{enumerate}
\item If $Q$ is a closed interval, then the multifractal spectrum has $0$ to $3$ phase transitions.
\item If $Q$ is a point $q \in \mathbb{R},$ then $\alpha^{+}=\alpha^{-}.$  Hence, the multifractal spectrum has $0$ to $2$ phase transitions.
\item If $Q$ is the half-open interval $(-\infty,q_{1}),$ then $\alpha_{\lim} \ge \alpha^{+}.$
The multifractal spectrum would then have $0$ to $1$ phase transition.  
\item If $Q$ is the half-open interval $(q_{0},\infty),$ then $\alpha^{-} \ge \alpha_{\lim}.$
The multifractal spectrum would then have $0$ to $1$ phase transition.  
\item If $Q$ is the open interval $(-\infty,\infty),$ then $\alpha^{+}=\alpha^{-}=\alpha_{\lim}=\alpha_{\inf}=\alpha_{\sup}.$  Then, the multifractal spectrum is constant because $f_{\mu}(\alpha)=t_{\infty}.$
\end{enumerate}

\begin{proposition}\label{prop:NO1}
Let $Q=\emptyset.$  This yields no phase transitions for the multifractal spectrum.
\end{proposition}

\begin{proof}
Since $Q=\emptyset,$ $T(q)>\tilde{t}(q)$ for all $q \in \mathbb{R}.$  Hence, $\mathcal{P}(q\phi-t\psi)$ is analytic as a function of $t$ on an $\varepsilon-$neighbourhood of $T(q).$  This gives us that $T(q)$ is analytic on all of $\mathbb{R}$ (as we proved earlier).  Thus, by Proposition \ref{prop:AA1} the proposition follows.  
\end{proof}

Therefore, we have proven Theorem \ref{thm:M1}.  Theorem \ref{thm:M1} tells us that when $0<\alpha_{\lim}<\infty,$ the multifractal spectrum has $0$ to $3$ phase transitions.  We repeat the theorem for completeness:
\begin{theorem}
Let $\phi: \Sigma \rightarrow \mathbb{R}^{-}$  be a potential with Gibbs measure $\mu$ such that $P(\phi)<\infty.$ and $\psi: \Sigma \rightarrow \mathbb{R}^{+}$ be a metric potential.  Assume that $\phi$ and $\psi$ are non-cohomologous locally H\"older potentials such that $\alpha_{\lim}<\infty.$   
\begin{enumerate}
\item There exist intervals $A_{i}$ such that $f_{\mu}(\alpha)$ is analytic on each of their interiors.  
\item The interval $(\alpha_{\inf}, \alpha_{\sup})=\cup_{i=1}^{j}A_{i}$ such that $j=\{1,2,3,4\}.$ 
\item The multifractal spectrum is concave on $(\alpha_{\inf}, \alpha_{\sup}),$ has its maximum at $\alpha(0),$ and has zero to three phase transitions.
\end{enumerate}
\end{theorem}

We now analyse the case when $\alpha_{\lim}=\infty.$

\begin{theorem}
Let $\phi: \Sigma \rightarrow \mathbb{R}^{-}$  be a potential with Gibbs measure $\mu$ such that $\mathcal{P}(\phi)<\infty$ and $\psi: \Sigma \rightarrow \mathbb{R}^{+}$ be a metric potential.  Assume that $\phi$ and $\psi$ are non-cohomologous locally H\"older potentials such that $\alpha_{\lim}=\infty.$   
\begin{enumerate}
\item There exist intervals $A_{i}$ such that $f_{\mu}(\alpha)$ is analytic on each of their interiors.  
\item The interval $(\alpha_{\inf}, \alpha_{\sup})=\cup_{i=1}^{j}A_{i}$ such that $j=\{1,2\}.$  
\item The multifractal spectrum is concave on $(\alpha_{\inf}, \alpha_{\sup}),$ is equal to its maximum $f_{\mu}(\alpha(0))$ on $(\alpha(0),\alpha_{\sup}),$ and has zero to one phase transition.
\end{enumerate}
\end{theorem}

\begin{proof}
Apply Propositions \ref{prop:ID1}, \ref{prop:A2}, and \ref{prop:NO1} because they do not need $\alpha_{\lim}<\infty.$
\end{proof}

We give a map that generates examples of phase transitions for $f_{\mu}(\alpha).$

\section{Adaptation For the Gauss Map}
Up until this point, we have only considered the multifractal spectrum with respect to a locally H\"older function, defined by a general expanding map.  We give a specific expanding map as follows.

\subsection{The Gauss Map}
\begin{definition}
The {\em Gauss map} $G: [0,1]\setminus{\mathbb{Q}} \rightarrow [0,1]\setminus{\mathbb{Q}}$ is defined by
$$G(x)=\frac{1}{x} \mod{1}.$$
\end{definition}

The inverse branches of $G$ are similar to a translated Gauss map.

\begin{definition}
Let $I_{b}:=[\frac{1}{b+1},\frac{1}{b}]\setminus{\mathbb{Q}}.$  Define $G_{b}:[0,1]\setminus{\mathbb{Q}} \rightarrow I_{b}$ as the {\em inverse branch} $(G|_{I_{b}})^{-1}$ of the Gauss map, which is
$$G_{b}(x)=\frac{1}{x+b}$$ 
for $x \in [0,1]\setminus{\mathbb{Q}}.$  For each $\tilde{b}=(b_{1},b_{2},...,b_{n}) \in \mathbb{N}^{n},$ the composition of these inverse branches is
$$G_{\tilde{b}}:=G_{b_{1}} \circ G_{b_{2}} \circ \cdots \circ G_{b_{n}}.$$
\end{definition}

From this point, we use the full shift  $\Sigma=\mathbb{N}^{\mathbb{N}}.$  The coding map we take between $\Sigma$ and $[0,1]$ is the continued fraction map.

\begin{definition}
The {\em coding map},
$\pi: \Sigma \rightarrow [0,1]\setminus{\mathbb{Q}},$ is defined as follows.  For each sequence $a \in \Sigma$ such that $a=(a_{1}(x),a_{2}(x),...),$ the map $\pi$ is
$$\pi(a):=\frac{1}{a_{1}(x)+\frac{1}{a_{2}(x)+\frac{1}{a_{3}(x)+\cdots}}}$$
such that 
\begin{enumerate}
\item $x=\pi(a)$
\item for each $i,k \in \mathbb{N},$ $G^{i-1}(x) \in (\frac{1}{k+1},\frac{1}{k})$ yields that $a_{i}(x)=k.$
\end{enumerate}
\end{definition}

We use the Gauss map to define the potential $\psi.$

\subsection{Thermodynamic Formalism}\hspace*{\fill} \par
By symbolic coding, we provide definitions for our locally H\"older potentials $\phi$ and $\psi.$

\begin{definition}
Define the locally H\"older potential $\psi:\Sigma \rightarrow \mathbb{R}^{+}$ on each  word $x \in \Sigma$ by
$$\psi(x)=\log|G'(\pi(x))|.$$
\end{definition}

Again, we assume that $\phi: \Sigma \rightarrow \mathbb{R}^{-}$ is a locally Holder potential on $\Sigma$ such that $0 \le \mathcal{P}(\phi)<\infty.$  Again, we assume that $\phi$ and $\psi$ are non-cohomologous to each other.  Using the potential $\psi,$ the following lemma relates the Gauss map to the diameter of any cylinder in $\Sigma.$  We use the following result to prove our main theorem.  

\begin{lemma}\label{lemma:D1}
Let $\mu$ be an ergodic $G-$invariant measure on $[0,1],$ $\nu$ be an ergodic $\sigma-$invariant measure on $\Sigma,$ and let $\psi(x)=\log|G'(\pi(x))|$ for each word $x \in \Sigma.$  Then, for $\mu-$a.e. $z=\pi(x) \in [0,1]$ such that $x=(x_{1},x_{2},...) \in \Sigma,$ 
$$-\int_{\Sigma} \psi\, \mathrm{d}\nu=-\int_{0}^{1} \log{|G'|}\, \mathrm{d}\mu=\lim\limits_{n \rightarrow \infty}-\frac{1}{n}\sum_{i=0}^{n-1}\log{|G'(G^{i}(z))|} $$  
$$\text{and} \lim\limits_{n \rightarrow \infty}-\frac{1}{n}\sum_{i=0}^{n-1}\log{|G'(G^{i}(z))|}=\lim\limits_{n \rightarrow \infty}-\frac{1}{n}\log|(G^{n})'(z)|=\lim\limits_{n \rightarrow \infty}\frac{1}{n}\log{|[x_{1},..,x_{n}]|}.$$
Furthermore, our choice of $\psi$ is a metric potential.
\end{lemma}

\begin{proof}
By the mean value theorem, there exists $z \in \pi([x_{1},...,x_{n}])$ such that $(G^{n})'(z)=(|[x_{1},...,x_{n}]|)^{-1}.$

By the chain rule, for $z \in [0,1],$ it follows that \[(G^{n})'(z)=G'(G^{n-1}(z))G'(G^{n-2}(z)) \cdots G'(z).\]

Combining both the diameter and chain rule equalities for $(G^{n})'(z),$ we get that
\begin{equation}\label{eq:1}
\sum_{i=0}^{n-1}\psi(\sigma^{i}(x))=\log{|(G^{n})'(z)|}=-\log{|[x_{1},...,x_{n}]|}=\sum_{i=0}^{n-1}\log{|G'(G^{i}(z))|}
\end{equation}
Hence,
\[|[x_{1},...,x_{n}]|=\exp\left(-\sum_{i=0}^{n-1}\psi(\sigma^{i}(x))\right).\]
Then, for $\mu-$a.e.$z=\pi(x) \in [0,1],$
the result follows from the Birkhoff ergodic theorem and $\psi$ is a metric potential.
\end{proof}

Since we have an expression for $\log{|[x_{1},...,x_{n}]|},$ let us again consider subsets of $\Sigma$ with symbolic dimension $\alpha.$

\section{The Multifractal Spectrum and $X_{\alpha}$}
We now revisit the set $X_{\alpha}^{s}$ and recall the definition of symbolic dimension.  Again, we take $\mu$ as the Gibbs measure for the potential $\phi: \Sigma \rightarrow \mathbb{R}^{-}.$  There is also a measure $\bar{\mu}=\mu \circ \pi^{-1}$ on $[0,1].$

\begin{definition}
For each $\alpha \in [\alpha_{\inf},\alpha_{\sup}],$ there exists a word $x \in [x_{1},...,x_{n}] \subset \Sigma$ has \text{\em symbolic dimension} $\alpha$ if
\[
d_{\mu}(x):=\lim\limits_{n \rightarrow \infty}\frac{\log{\mu([x_{1},...,x_{n}])}}{\log{|[x_{1},...,x_{n}]|}}=\alpha.
\]
\end{definition}

Define the level set
$$\pi(X_{\alpha}^{s})=\pi\left(\left\{x \in \Sigma: \lim\limits_{n \rightarrow \infty}\frac{\log{\mu([x_{1},...,x_{n}])}}{\log{|[x_{1},...,x_{n}]|}}=\alpha\right\}\right).$$

There is a similar type of local dimension for $[0,1].$
\begin{definition}
For $r>0,$ let $B(x,r)$ be a ball centered at $x \in [0,1].$  Then, for each $\alpha \in [\alpha_{\inf},\alpha_{\sup}],$ $x \in [0,1]$ has \text{\em local dimension} $\alpha$ if
\[
\lim\limits_{r \rightarrow 0}\frac{\log{\bar{\mu}(B(x,r))}}{\log{r}}=\alpha.
\]
\end{definition}

Define the set
$$X_{\alpha}=\{x \in [0,1]\setminus{\mathbb{Q}}: \lim\limits_{r \rightarrow 0}\frac{\log{\bar{\mu}(B(x,r))}}{\log{r}}=\alpha\}.$$ 

We now give an important result by Iommi \cite{iommi2005multifractal} (Page 1891, Theorem 3.7).

\begin{proposition}\label{prop:I2}
For each $\alpha \in (\alpha_{\inf}, \alpha_{\sup}),$ $\pi(X_{\alpha}^{s})$ satisfies
$$\dim_{H}(\pi(X_{\alpha}^{s}))=f_{\mu}(\alpha)=\inf\limits_{q \in \mathbb{R}}\{T(q)+q\alpha\}.$$
\end{proposition}

We have a similar result for $X_{\alpha}.$

\begin{proposition}\label{thm:G1}
For each $\alpha \in (\alpha_{\inf}, \alpha_{\sup}),$ we have that
$$f_{\mu}(\alpha)=\dim_{H}(X_{\alpha}).$$
\end{proposition}

We outline the proof of Proposition \ref{thm:G1} as follows.  To prove $\dim_{H}(X_{\alpha}) \ge f_{\mu}(\alpha),$ we first notice that the geometric structure of the Gauss map makes it difficult to cover cylinder sets with neighbourhoods and vice versa.  Define $T_{n}(q)$ as the temperature function in the setting of $\Sigma_{n}.$  Take $\mu_{q_{n}}$ as the equilibrium state for $q\phi-T_{n}(q)\psi.$  We provide a condition involving the hitting times of a $\mu_{q_{n}}-$typical cylinder set.  This gives us that $X_{\alpha}^{s} \subset \pi^{-1}(X_{\alpha}).$  Then, we use an inequality proven by Pesin and Weiss \cite{pesin1997multifractal} involving the pointwise dimension of $x \in \Sigma_{n}$ and the multifractal spectrum $f_{\mu,n}(\alpha).$  To apply these results to $\Sigma,$ we approximate pressure in $\Sigma$ with pressure in $\Sigma_{n}.$  This gives us a monotone convergence argument for $T_{n}(q)$ and it induces a similar argument involving Pesin and Weiss's inequality.  Hence, the result follows.\par
To prove $\dim_{H}(X_{\alpha}) \le f_{\mu}(\alpha),$ we again use the behaviour of the Gauss map.  Neighbouring $m-$cylinders nearly have equal diameters for large $m \in \mathbb{N}.$  We exploit this fact by slightly increasing the size of cylinders around each $x \in X_{\alpha}^{s}.$  This creates a Hausdorff cover for $\pi^{-1}(X_{\alpha}),$ which we use in an argument for bounding the Hausdorff measure of $X_{\alpha}.$  This gives the appropriate upper bound for $\dim_{H}(X_{\alpha}).$

Through Proposition \ref{thm:G1}, we can apply the results of Theorems \ref{thm:M1} and \ref{thm:M2} to the potential $\psi=\log|G'|.$  Now, we consider examples in the next section. 

\section{Examples of Phase Transitions for the Gauss Map}

We consider the geometric potential $\psi=\log|G'|$ and provide examples in which the multifractal spectrum has up to infinitely many phase transitions.
Also, we apply the results of Theorems \ref{thm:M1} and \ref{thm:M2} as well as provide an example when $\alpha_{\lim}$ does not exist.  Furthermore, we remark that by estimating $\psi(x)$ for $x=(x_{1},x_{2},...)$ with a locally H\"older potential $2\log(x_{1}),$ we find that $\psi$ is unbounded.  Now, we consider the first of our three cases: $\alpha_{\lim}<\infty.$

\subsection{The Case $\alpha_{\lim}<\infty$}\hspace*{\fill} \par
As reflected by Theorem \ref{thm:M1}, we provide examples of potentials such that $f_{\mu}(\alpha)$ has zero to three phase transitions.  The easiest way to show this is by approximating $\phi$ and $\psi$ with locally constant potentials.   

\subsubsection{Approximation with Locally Constant Potentials}\hspace*{\fill} \par
Since $\phi$ and $\psi$ are locally H\"older, we will give some of our estimates in terms of one periodic sequences.  Now, we provide a critical technique used in our examples.  Locally H\"older functions can be approximated using locally constant functions as follows.  Let $\Sigma=\mathbb{N}^{\mathbb{N}},$ take $\psi(x)=\log|G'(\pi(x))|$ for each $x \in \Sigma,$ and assume that $\tilde{\psi}(x)$ is a locally constant function.

Then, consider $\psi-\bar{\psi}.$  This potential is locally H\"older because it is a difference of locally H\"older functions.  For any arbitrary $n-$cylinder $[x_{1},...,x_{n}]$ in $\Sigma$ and $x \in [x_{1},...,x_{n}],$
\[0 \le |\psi(x)-\bar{\psi}(x)| \le V_{n}(\psi-\tilde{\psi}) \le C\theta^{n}\]
for some $C>0$ and $\theta \in (0,1).$
Let $\varepsilon>0.$  Since $C>0$ is fixed, we have that for $n \ge N$ such that $N \in \mathbb{N}$ is large,
\[0 \le |\psi(x)-\bar{\psi}(x)| \le V_{n}(\psi-\bar{\psi}) \le C\theta^{N} \le \varepsilon.\]
It follows that we can approximate $\psi$ with a locally H\"older and locally constant potential $\bar{\psi}(x)=-\log(\frac{6}{\pi^{2}x_{1}^{2}}).$

Furthermore, we provide a formula used to calculate the topological pressure of $q\bar{\phi}-t\bar{\psi}.$  Let $0<p_{i}<1$ and $0<s_{i}<1$ for each $i \in \mathbb{N}.$  We will usually take $\bar{\phi}(x)=\log{p_{x_{1}}}$ and $\bar{\psi}(x)=\log{s_{x_{1}}^{-1}}.$  For most of our examples, $s_{i}=\frac{6}{\pi^{2}i^{2}}.$

Then, for $\varepsilon>0,$
\[\log\left(\sum_{i=1}^{\infty}p_{i}^{q}s_{i}^{t}\right) - \varepsilon \le \mathcal{P}(q\bar{\phi}-t\bar{\psi}) \le \log\left(\sum_{i=1}^{\infty}p_{i}^{q}s_{i}^{t}\right) + \varepsilon.\]

We now proceed with our examples.

\subsubsection{Example of Zero Phase Transitions}\hspace*{\fill} \par
For each $i \in \mathbb{N},$ let $0<p_{i}<1.$  Take $p_{i}=\frac{C}{(i+1)^{3}}$ with $C \approx 4.9491$ chosen so that $\sum_{i=1}^{\infty}p_{i}=1.$
For each $x=(x_{1},...) \in \Sigma,$ define the potentials $\phi:\Sigma\rightarrow\mathbb{R}^{-}$ and $\psi:\Sigma\rightarrow\mathbb{R}^{+}$ as follows:
\[
\phi(x)=\log{p_{x_{1}}}=\log\left(\frac{4.9491}{(x_{1}+1)^{3}}\right)
\]
and
\[
\psi(x)=\log|G'(\pi(x))|.
\]
Throughout this example, we will estimate $\psi(x)$ with $\tilde{\psi}(x)=-\log(\frac{6}{\pi^{2}(x_{1})^{2}})$ and take $\tilde{\phi}=\phi.$

We now consider the potential $q\tilde{\phi}-t\tilde{\psi}.$
As noticed earlier, the number of phase transitions is determined by the relationship between $T(q)$ and $\tilde{t}(q).$  We found that $\tilde{t}(q)=-\alpha_{\lim}q+t_{\infty}.$
The calculations for $\alpha_{\lim}$ are as follows:
Let $\bar{j}=(j,j,...)$ such that $j \in \mathbb{N}.$  Then,
\begin{equation}\label{eq:K}
\alpha_{\lim}=\lim\limits_{j\rightarrow\infty}\frac{\tilde{\phi}(\bar{j})}{-\tilde{\psi}(\bar{j})}=\lim\limits_{j\rightarrow\infty}\frac{-3\log{j}}{-2\log{j}}=\frac{3}{2}.
\end{equation}
The value $t_{\infty}$ can be found using Lemma~\ref{prop:L1}. \par
Remember that 
\[
t_{\infty}=\inf\{t\in\mathbb{R}:\mathcal{P}(-t\tilde{\psi})<\infty\}= \inf\{t\in\mathbb{R}:Z_{1}(-t\tilde{\psi})<\infty\}.
\]
Let $j \in \mathbb{N}.$  We let $V_{1,j}(\tilde{\psi})$ be the H\"older constant for $\tilde{\psi}$ such that for every $i,x \in [j],$
\[|\tilde{\psi}(i)-\tilde{\psi}(x)| \le V_{1,j}(\tilde{\psi}).\]
Thus, we must find the infimum of all $t\in\mathbb{R}$ such that
\begin{eqnarray}\label{eq:L}
Z_{1}(-t\tilde{\psi}) &=& \sum_{x_{1}=1}^{\infty}\exp\sup\limits_{x\in[x_{1}]}\left(-2t\log{x_{1}}-t\log\left(\frac{6}{\pi^{2}}\right)\right) \nonumber \\
&\le&
\sum_{j=1}^{\infty}\exp\left(\log(j^{-2t})-t\log\left(\frac{6}{\pi^{2}}\right)+tV_{1,j}(\tilde{\psi})\right)<\infty 
\end{eqnarray}
for $t>\frac{1}{2}$ because $V_{1,j}(\tilde{\psi}) \rightarrow 0$ as $j \rightarrow \infty.$  Hence, $t_{\infty}=\frac{1}{2}.$  Thus, by Equations (\ref{eq:K}) and (\ref{eq:L}),
\[
\tilde{t}(q)=-\frac{3}{2}q+\frac{1}{2}.
\]

By definition, we must have that $Z_{n}(q\tilde{\phi}-t\tilde{\psi})<1$ in order for $\mathcal{P}(q\tilde{\phi}-t\tilde{\psi})<0.$  As noticed earlier, we get $Q$ by considering the values of $q\in\mathbb{R}$ satisfy $T(q)=\tilde{t}(q).$  Hence, we must find $q\in\mathbb{R}$ such that $Z_{n}(q\tilde{\phi}-\tilde{t}(q)\tilde{\psi})<1.$  It is enough to consider $q\in\mathbb{R}$ such that $Z_{1}(q\tilde{\phi}-\tilde{t}(q)\tilde{\psi})<1.$ \par
Let $j \in \mathbb{N}.$  We let $V_{1,j}(q\tilde{\phi}-\tilde{t}(q)\tilde{\psi})$ be the H\"older constant for $q\tilde{\phi}-\tilde{t}(q)\tilde{\psi}$ such that for every $i,x \in [j],$
\[|(q\tilde{\phi}-\tilde{t}(q)\tilde{\psi})(i)-(q\tilde{\phi}-\tilde{t}(q)\tilde{\psi})(x)| \le V_{1,j}(q\tilde{\phi}-\tilde{t}(q)\tilde{\psi}).\]
Furthermore, for each $x \in [j]$ and $\bar{j}=(j,j,j....),$
\[|\sup\limits_{x\in[x_{1}]}q\tilde{\phi}-\tilde{t}(q)\tilde{\psi})(x)-(q\tilde{\phi}-\tilde{t}(q)\tilde{\psi})(\bar{j})|=C_{1,j}(q\tilde{\phi}-\tilde{t}(q)\tilde{\psi})\]
such that $C_{1,j}(q\tilde{\phi}-\tilde{t}(q)\tilde{\psi})>0.$

Hence,
\begin{eqnarray}
Z_{1}(q\tilde{\phi}-\tilde{t}(q)\tilde{\psi}) &=& \sum_{x_{1}=1}^{\infty}\exp\sup\limits_{x\in[x_{1}]}(q\tilde{\phi}-\tilde{t}(q)\tilde{\psi})(x) \nonumber \\
&=&
\sum_{j=1}^{\infty}\exp((q\tilde{\phi}-\tilde{t}(q)\tilde{\psi})(\bar{j})+C_{1,j}(q\tilde{\phi}-\tilde{t}(q)\tilde{\psi})) \nonumber \\
&=&
\sum_{j=1}^{\infty}\frac{(4.9491)^{q}(\frac{6}{\pi^{2}})^{-\frac{3}{2}q+\frac{1}{2}}}{j^{3q+2\tilde{t}(q)}}\exp(C_{1,j}(q\tilde{\phi}-\tilde{t}(q)\tilde{\psi}))=\infty. \nonumber
\end{eqnarray}
Hence, $T(q)\not=\tilde{t}(q)$ for any $q\in\mathbb{R}.$  Thus, $Q=\emptyset.$
Therefore, by Proposition 6.3, the multifractal spectrum has no phase transitions.

\subsubsection{One Phase Transition}\hspace*{\fill} \par
We first provide pictures of the multifractal spectrum and $T(q)$ for the following example.  

\begin{figure}[H]
\centering
\begin{subfigure}{.5\textwidth}
  \centering
  \includegraphics[width=1\linewidth]{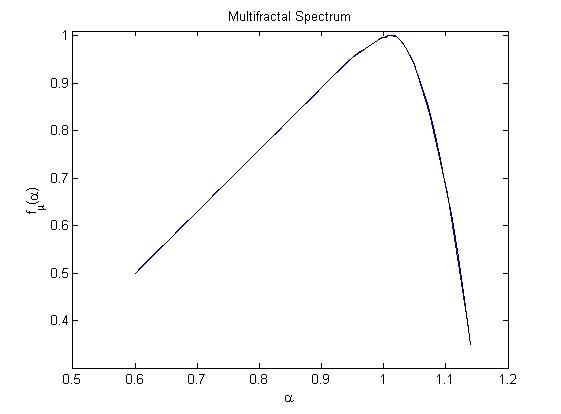}
  \caption{\text{Multifractal Spectrum on} $(0.6,1.13653)$}
  \label{fig:sub1}
\end{subfigure}%
\begin{subfigure}{.5\textwidth}
  \centering
  \includegraphics[width=1\linewidth]{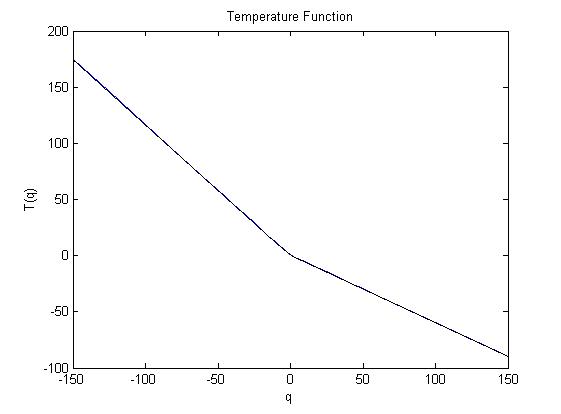}
  \caption{Temperature Function on $(-150,150)$}
  \label{fig:sub2}
\end{subfigure}
\caption{We observe that $\alpha_{\lim}=0.6$ and $T(q)$ has phase transition around $q=1.3.$ Most importantly, the multifractal spectrum has its phase transition around $\alpha \in (0.955,0.999)$ and its maximum is $f_{\mu}(\alpha(0))=1$ (such that $\alpha(0)=1.0068$).}
\label{fig:test}
\end{figure}

For each $i \in \mathbb{N},$ let $0<p_{i}<1.$  Take $p_{i}=\frac{C}{(i)^{\frac{6}{5}}(\log(i+2))^{2}}$ with $C \approx 0.67569$ chosen so that $\sum_{i=1}^{\infty}p_{i}=1.$
For each $x=(x_{1},...) \in \Sigma,$ define the potentials $\phi:\Sigma\rightarrow\mathbb{R}^{-}$ and $\psi:\Sigma\rightarrow\mathbb{R}^{+}$ as follows:
\[
\phi(x)=\log{p_{x_{1}}}=\log\left(\frac{0.67569}{(i)^{\frac{6}{5}}(\log(i+2))^{2}}\right)
\]
and
\[
\psi(x)=\log|G'(\pi(x))|.
\]

Hence, we define the locally constant potentials $\bar{\phi}$ and $\bar{\psi}$ as follows:
\[\bar{\phi}(x):=\log{p_{x_{1}}} \text{ and } \bar{\psi}(x):=-\log\left({\frac{6}{\pi^{2}{x_{1}}^{2}}}\right)\] such that
\[p_{i}=\frac{0.67569}{(i)^{\frac{6}{5}}(\log(i+2))^{2}}\]
for each $i \in \mathbb{N}.$

We will again find an explicit expression for $\tilde{t}(q).$  Let $\bar{j}=(j,j,...)$ for any $j \in \mathbb{N}.$  We have 
\[
\alpha_{\lim}=\lim\limits_{j\rightarrow\infty}\frac{\bar{\phi}(\bar{j})}{-\bar{\psi}(\bar{j})}=\lim\limits_{j\rightarrow\infty}\frac{-\frac{6}{5}\log{j}}{-\log|G'(\pi(\bar{j}))|}=\lim\limits_{j\rightarrow\infty}\frac{-\frac{6}{5}\log{j}}{-2\log{j}}=\frac{3}{5}.
\]
Again, we have that $Z_{1}(-t\bar{\psi})<\infty$ for $t>\frac{1}{2},$ so
$t_{\infty}=\frac{1}{2}.$  Hence, we get that
\[
\tilde{t}(q)=-\frac{3}{5}q+\frac{1}{2}.
\]

We prove that it is possible for the set $Q$ to equal $[q_{0},\infty)$ for some $q_{0}>\frac{6}{5}.$  This involves using partial sums of $\mathcal{P}(q\bar{\phi}-\tilde{t}(q)\bar{\psi})$ to bound $\mathcal{P}(q\phi-\tilde{t}(q)\psi)$ below for some  $q<\frac{6}{5}$ and $Z_{1}(q\bar{\phi}-\tilde{t}(q)\bar{\psi})$  to estimate $\mathcal{P}(q\phi-\tilde{t}(q)\psi)$ above for a fixed $q>\frac{6}{5}.$ First, we need to prove that there exists a $q>\frac{6}{5}$ such that $Z_{1}(q\bar{\phi}-\tilde{t}(q)\bar{\psi})<1.$  Again, we let $V_{1,j}(q\bar{\psi}+-\tilde{t}(q)\bar{\psi})$ be the H\"older constant for $q\bar{\phi}-\tilde{t}(q)\bar{\psi}$ such that for every $i,x \in [j],$
\[|(q\bar{\phi}-\tilde{t}(q)\bar{\psi})(i)-(q\bar{\phi}-\tilde{t}(q)\bar{\psi})(x)| \le V_{1,j}(q\bar{\phi}-\tilde{t}(q)\bar{\psi}).\]
The preceding estimate gives us that
\begin{eqnarray}
Z_{1}(q\bar{\phi}-\tilde{t}(q)\bar{\psi}) &=& \sum_{x_{1}=1}^{\infty}\exp\sup\limits_{x\in[x_{1}]}(q\bar{\phi}-\tilde{t}(q)\bar{\psi})(x) \nonumber \\
& \le & 
\sum_{j=1}^{\infty}\exp(q\bar{\phi}-\tilde{t}(q)\bar{\psi})(\bar{j})+V_{1,j}(q\bar{\phi}-\tilde{t}(q)\bar{\psi})) \nonumber \\
&<& \sum_{j=1}^{\infty}\frac{(0.67569)^{q}(\frac{6}{\pi^{2}})^{-\frac{3}{5}q+\frac{1}{2}}}{(j)(\log(j+2))^{2q}}\exp(V_{1,j}(q\bar{\phi}-\tilde{t}(q)\bar{\psi}))). \nonumber
\end{eqnarray}
\par We notice the following:  Let $\varepsilon>0$ and $N \in \mathbb{N}$ be large.  Then,
\[\sum_{j=1}^{\infty}\frac{(0.67569)^{q}(\frac{6}{\pi^{2}})^{-\frac{3}{5}q+\frac{1}{2}}}{(j)(\log(j+2))^{2q}}=\sum_{j=1}^{N-1}\frac{(0.67569)^{q}(\frac{6}{\pi^{2}})^{-\frac{3}{5}q+\frac{1}{2}}}{(j)(\log(j+2))^{2q}}+\sum_{j=N}^{\infty}\frac{(0.67569)^{q}(\frac{6}{\pi^{2}})^{-\frac{3}{5}q+\frac{1}{2}}}{(j)(\log(j+2))^{2q}}.\]
Let $f(q)=\sum_{j=N}^{\infty}\frac{(0.67569)^{q}(\frac{6}{\pi^{2}})^{-\frac{3}{5}q+\frac{1}{2}}}{(j)(\log(j+2))^{2q}}=\sum_{j=N}^{\infty}\frac{((0.67569)(\frac{6}{\pi^{2}})^{-\frac{3}{5}})^{q}(\frac{6}{\pi^{2}})^{\frac{1}{2}}}{(j)(\log(j+2))^{2q}}.$
We have that $f'(q)=$
\[\sum_{j=N}^{\infty}\frac{((0.67569)(\frac{6}{\pi^{2}})^{-\frac{3}{5}})^{q}(\frac{6}{\pi^{2}})^{\frac{1}{2}}\log((0.67569)(\frac{6}{\pi^{2}})^{-\frac{3}{5}})}{(j)(\log(j+2))^{2q}}-2\sum_{j=N}^{\infty}\frac{((0.67569)(\frac{6}{\pi^{2}})^{-\frac{3}{5}})^{q}(\frac{6}{\pi^{2}})^{\frac{1}{2}}}{j\log\log(j+2)(\log(j+2))^{2q}}\]
which is less than $\varepsilon$ for a fixed $q > \frac{6}{5}.$  Hence, $\sum_{j=N}^{\infty}\frac{(0.67569)^{q}(\frac{6}{\pi^{2}})^{-\frac{3}{5}q+\frac{1}{2}}}{(j)(\log(j+2))^{2q}}$ is a decreasing function with respect to $q.$ \par
Let $g(q)=\sum_{j=1}^{N-1}\frac{(0.67569)^{q}(\frac{6}{\pi^{2}})^{-\frac{3}{5}q+\frac{1}{2}}}{(j)(\log(j+2))^{2q}}.$
Similar to what was shown earlier, $g'(q)<0.$
Hence,
\[\sum_{j=1}^{N-1}\frac{(0.67569)^{q}(\frac{6}{\pi^{2}})^{-\frac{3}{5}q+\frac{1}{2}}}{(j)(\log(j+2))^{2q}} \text{ and } \sum_{j=N}^{\infty}\frac{(0.67569)^{q}(\frac{6}{\pi^{2}})^{-\frac{3}{5}q+\frac{1}{2}}}{(j)(\log(j+2))^{2q}}\]
are decreasing functions with respect to $q.$
Then, for a fixed $q>\frac{6}{5},$
\begin{equation}\label{eq:N}
\sum_{j=1}^{N-1}\frac{(0.67569)^{q}(\frac{6}{\pi^{2}})^{-\frac{3}{5}q+\frac{1}{2}}}{(j)(\log(j+2))^{2q}} \le \int_{1}^{N-1} \frac{(0.67569)^{q}(\frac{6}{\pi^{2}})^{-\frac{3}{5}q+\frac{1}{2}}}{(j)(\log(j+2))^{2q}}\, \mathrm{d}\mu \le 1-\varepsilon
\end{equation}
\begin{equation}\label{eq:P}
\text{and} \sum_{j=N}^{\infty}\frac{(0.67569)^{q}(\frac{6}{\pi^{2}})^{-\frac{3}{5}q+\frac{1}{2}}}{(j)(\log(j+2))^{2q}} \le \int_{N}^{\infty} \frac{(0.67569)^{q}(\frac{6}{\pi^{2}})^{-\frac{3}{5}q+\frac{1}{2}}}{(j)(\log(j+2))^{2q}}\, \mathrm{d}\mu \le \varepsilon.
\end{equation}
Therefore, by Equations (\ref{eq:N}) and (\ref{eq:P}),
\[Z_{1}(q\bar{\phi}-\tilde{t}(q)\bar{\psi})<\sum_{j=1}^{\infty}\frac{(0.67569)^{q}(\frac{6}{\pi^{2}})^{-\frac{3}{5}q+\frac{1}{2}}}{(j)(\log(j+2))^{2q}}\exp(V_{1}(q\bar{\phi}-\tilde{t}(q)\bar{\psi}))) \le 1\]
for a fixed $\widehat{q}>\frac{6}{5}$ (with $\widehat{q}$ not necessarily equal to $q_{0}$).

Since ${\bar{\phi}}$ and ${\bar{\psi}}$ approximate $\phi$ and $\psi$ respectively, we will prove that
there exists a value $1<q<\frac{6}{5}$ such that
\[\mathcal{P}(q\bar{\phi}-\tilde{t}(q)\bar{\psi})>0.\]
Let $\varepsilon>0.$  It follows that
\[\mathcal{P}(q\bar{\phi}-\tilde{t}(q)\bar{\psi})>\sum_{j=1}^{\infty}\frac{(0.67569)^{q}(\frac{6}{\pi^{2}})^{-\frac{3}{5}q+\frac{1}{2}}}{(j)(\log(j+2))^{2q}}-\varepsilon\]
because $\bar{\psi}$ and $\bar{\phi}$ are locally constant.  We will estimate $\mathcal{P}(q\bar{\phi}-\tilde{t}(q)\bar{\psi})$ by using 
\[\sum_{j=1}^{\infty}\frac{(0.67569)^{q}(\frac{6}{\pi^{2}})^{-\frac{3}{5}q+\frac{1}{2}}}{(j)(\log(j+2))^{2q}}.\]

In particular, we get that
\[\sum_{j=1}^{\infty}\frac{(0.67569)^{q}(\frac{6}{\pi^{2}})^{-\frac{3}{5}q+\frac{1}{2}}}{(j)(\log(j+2))^{2q}}>\sum_{j=1}^{25}\frac{(0.67569)^{q}(\frac{6}{\pi^{2}})^{-\frac{3}{5}q+\frac{1}{2}}}{(j)(\log(j+2))^{2q}}>1\]
for $q=1.15.$  Since 
\begin{equation}\label{eq:partialsum}
\sum_{j=1}^{N}\frac{(0.67569)^{q}(\frac{6}{\pi^{2}})^{-\frac{3}{5}q+\frac{1}{2}}}{(j)(\log(j+2))^{2q}}
\end{equation}
is a decreasing function with respect to $q$ and the value of sum (\ref{eq:partialsum}) increases 
as $N \rightarrow \infty,$ there exists a value $1<q<\frac{6}{5}$ such that
\[\mathcal{P}(q{\bar{\phi}}+-\tilde{t}(q){\bar{\psi}})>0.\]
Since ${\bar{\phi}}$ and ${\bar{\psi}}$ approximate $\phi$ and $\psi$ respectively, we get that
there exists a value $1<q<\frac{6}{5}$ such that
\[\mathcal{P}(q\phi-\tilde{t}(q)\psi)>0.\]

This means that $T(q)=\tilde{t}(q)$ for all $\widehat{q}>\frac{6}{5}$ and $T(q)>\tilde{t}(q)$ for some $q<\frac{6}{5}.$  Thus, we have established that $Q=[q_{0},\infty)$ for $q_{0}>\frac{6}{5}.$  Now, we can consider $\alpha(q_{0})$ for $q_{0}>\frac{6}{5}.$  Using our earlier estimates, we have that
\begin{eqnarray}
\int -\phi\,\mathrm{d}\mu_{q} &\le&
\int -\bar{\phi}\,\mathrm{d}\mu_{q}=\sum_{j=1}^{\infty}-\frac{(0.67569)^{q}(\frac{6}{\pi^{2}})^{T(q)}}{j^{\frac{6}{5}q+2T(q)}(\log(j+2))^{2q}}\log\left(\frac{0.67569}{j^{\frac{6}{5}}(\log(j+2))^{2}}\right)\nonumber \\
&<&\sum_{j=1}^{\infty}\frac{1}{j^{\frac{6}{5}q+2T(q)-1}(\log(j+2))^{2q-1}}<\infty \nonumber
\end{eqnarray}
and
\begin{eqnarray}
\int \psi\, \mathrm{d}\mu_{q} &\le&
\int \bar{\psi}\, \mathrm{d}\mu_{q} =-\sum_{j=1}^{\infty}\frac{(0.67569)^{q}(\frac{6}{\pi^{2}})^{T(q)}}{j^{\frac{6}{5}q+2T(q)}(\log(j+2))^{2q}}\log\left(\frac{6}{\pi^{2}j^{2}}\right)  \nonumber \\
&<&\sum_{j=1}^{\infty}\frac{1}{j^{\frac{6}{5}q+2T(q)-1}(\log(j+2))^{2q-1}}<\infty \nonumber 
\end{eqnarray}
when $q>\frac{6}{5}.$  We remark that both integrals are infinite if $q<\frac{5}{6}.$  Hence, if $q_{0}>q>\frac{6}{5},$ 
\[
\alpha(q_{0})=\lim\limits_{q\rightarrow q_{0}^{-}}\frac{\int \phi\, \mathrm{d}\mu_{q}}{\int -\psi\, \mathrm{d}\mu_{q}}>\alpha_{\lim}.
\]

Using the techniques from the proof of Theorem~\ref{thm:M1}, we get the following.  Given that $q_{0}>\frac{6}{5},$ $\alpha(q_{0}) \ge \alpha_{\lim}=\alpha_{\inf}.$  In this case, $f_{\mu}(\alpha)$ is analytic on $(\alpha_{\inf}, \alpha(q_{0}))$ and $(\alpha(q_{0}),\alpha_{\sup}).$  The multifractal spectrum is an increasing linear function and equals $T(q_{0})+q_{0}\alpha$ on $(\alpha_{\inf},\alpha(q_{0})),$ is strictly concave on $(\alpha(q_{0}),\alpha_{\sup}),$ and has its maximum at $\alpha(0).$  The multifractal spectrum has its only phase transition at $\alpha(q_{0}).$ 

\subsubsection{Two and Three Phase Transitions}\hspace*{\fill} \par
For each $j \in \mathbb{N},$ let $0<p_{j}<1.$  Pick a $k\in\mathbb{N}$ and let \[p_{k}=\frac{C_{k}}{(k)^{\frac{6}{5}}(\log(k+2))^{2}}\] such that $C_{k}>1.$  For each $i \not= k \in \mathbb{N},$ take \[p_{i}=\frac{C}{(i)^{\frac{6}{5}}(\log(i+2))^{2}}\] with $0<C<1$ chosen so that $\sum_{i=1}^{\infty}p_{i}=1.$  For simplicity, we take $C=\frac{3}{5}$ because the introduction of $C_{k}$ must mean that $C<0.67569.$  Note that $p_{k}=\frac{C_{k}}{C}\frac{C}{k^{\frac{6}{5}}(\log(k+2))^{2}}.$  For each $x=(x_{1},...) \in \Sigma,$ define the potentials $\phi:\Sigma\rightarrow\mathbb{R}^{-}$ and $\psi:\Sigma\rightarrow\mathbb{R}^{+}$ as follows:
\[
\phi(x)=\log{p_{x_{1}}}=\log\left(\frac{C}{(i)^{\frac{6}{5}}(\log(i+2))^{2}}\right)
\]
if $x_{1}=i\not=k,$
\[
\phi(x)=\log{p_{x_{1}}}=\log\left(\frac{C_{k}}{(k)^{\frac{6}{5}}(\log(k+2))^{2}}\right)
\]
if $x_{1}=k,$ and
\[
\psi(x)=\log|G'(\pi(x))|.
\]
\par Let $\bar{j}=(j,j,...)$ for any $j \in \mathbb{N}.$
Again, we approximate $\phi$ and $\psi$ with locally constant functions $\bar{\phi}$ and $\bar{\psi}.$  
We let $\bar{\phi}(x)=\phi(x)$ and
$\bar{\psi}(x)=-\log\left(\frac{6}{\pi^{2}x_{1}^{-2}}\right)$ for each $x \in \Sigma.$  We again have that $Z_{1}(-t\psi)<\infty$ for $t>\frac{1}{2},$ so 
$\alpha_{\lim}=\frac{3}{5}$ and $t_{\infty}=\frac{1}{2}.$  Hence,
\[
\tilde{t}(q)=-\frac{3}{5}q+\frac{1}{2}.
\]
We will try to prove that it is possible for the set $Q$ to equal $[q_{0},q_{1}]$ for $q_{0}>\frac{5}{6}.$

Since the arguments are nearly identical to the previous example, we instead give an outline.  This involves using $\bar{\phi}$ and $\bar{\psi}$ to estimate $\mathcal{P}(q\phi-\tilde{t}(q)\psi)$ below for $q<\frac{5}{6},$ $Z_{1}(q\bar{\phi}-\tilde{t}(q)\bar{\psi})$ to estimate $\mathcal{P}(q\phi-\tilde{t}(q)\psi)$ above for $q_{0}>q>\frac{5}{6},$ and again $\bar{\phi}$ and $\bar{\psi}$ to estimate $\mathcal{P}(q\phi-\tilde{t}(q)\psi)$ below for $\frac{5}{6}<q<q_{1}<\infty.$  In the case that $Q=[q_{0},q_{1}]$ such that $q_{0}<\frac{5}{6}<q_{1},$ the only modification to the previous argument is the need to use $Z_{1}(q\bar{\phi}-\tilde{t}(q)\bar{\psi})$ to estimate $\mathcal{P}(q\bar{\phi}-\tilde{t}(q)\bar{\psi})$ above for $q<\frac{5}{6}.$  For some $q<\frac{5}{6},$ $Z_{1}(q\bar{\phi}-\tilde{t}(q)\bar{\psi})<1$ so $\mathcal{P}(q\bar{\phi}-\tilde{t}(q)\bar{\psi})<0.$

To prove that $T(q)$ has a second phase transition at $q_{1},$ we consider the following.  Using simple analysis, we find that we need to satisfy
\[\frac{\log{p_{k}}}{-2\log(k)}>\alpha_{\lim}\] in order for $q_{1}$ to exist (because this would mean that there exists $q>q_{1}$ such that $\alpha(q)>\alpha_{\lim}$).
Let $\varepsilon>0.$  Since $\alpha_{\lim}=\lim\limits_{j \rightarrow \infty}\frac{\log{p_{j}}}{\log|G'(\pi(\bar{j}))|}=\lim\limits_{j \rightarrow \infty}\frac{\log{p_{j}}}{2\log(j)},$ it follows that for each $j \ge J$ for some large $J \in \mathbb{N},$
\[\alpha_{\lim}-\varepsilon \le \frac{\log{p_{j}}}{-2\log{j}} \le \alpha_{\lim}+\varepsilon.\]
Thus, for $j \ge k,$ we must satisfy
\[\frac{\log{p_{k}}}{-2\log(k)}> \frac{\log{p_{j}}}{-2\log{j}}.\]
Hence, choosing $C_{k}>1$ gives us the second phase transition for $T(q).$

Now, we must analyse the behaviour of $\alpha(q)$ at $q_{0}$ and $q_{1}.$  The work to show that \[\int -\bar{\phi}\, \mathrm{d}\mu_{q}<\infty \text{ and} \int \bar{\psi}\, \mathrm{d}\mu_{q}<\infty\] for $q>\frac{5}{6}$ is identical to the previous example.  Hence, if $q<q_{0}\le \frac{5}{6},$  \[\int -\bar{\phi}\, \mathrm{d}\mu_{q}=\infty \text{ and} \int \bar{\psi}\, \mathrm{d}\mu_{q}=\infty.\]
The same results are true for $q_{1}.$ \par
Therefore, using the techniques from the proof of Theorem~\ref{thm:M1}, we get either two or three phase transitions for our chosen $\phi$ and $\psi,$ depending on the values of $q_{0}$ and $q_{1}:$
\begin{enumerate}
\item If $q_{0} \le \frac{5}{6} < q_{1},$
$\alpha(q_{0})=\alpha_{\lim}>\alpha(q_{1}).$  The multifractal spectrum is analytic on $(\alpha_{\inf},\alpha(q_{1})),$ $(\alpha(q_{1}),\alpha(q_{0})),$ and $(\alpha(q_{0}),\alpha_{\sup}).$  Furthermore, $f_{\mu}(\alpha)$ is strictly concave on $(\alpha_{\inf},\alpha(q_{1}))$ as well as $(\alpha(q_{0}),\alpha_{\sup}),$ is linear and equals $T(q_{1})+q_{1}\alpha$ on $(\alpha(q_{1}),\alpha(q_{0})),$ and has its maximum at $\alpha(0).$  The multifractal spectrum has phase transitions at $\alpha_{\lim}$ and $\alpha(q_{1})$ in this case.
\item  If $\frac{5}{6}<q_{0}<q_{1},$
$\alpha(q_{0})>\alpha_{\lim}>\alpha(q_{1}).$  The multifractal spectrum is analytic on $(\alpha_{\inf},\alpha(q_{1})),$ $(\alpha(q_{1}),\alpha_{\lim}),$ $(\alpha_{\lim},\alpha(q_{0})),$ and $(\alpha(q_{0}),\alpha_{\sup}).$  Furthermore, $f_{\mu}(\alpha)$ is strictly concave on $(\alpha_{\inf},\alpha(q_{1}))$ as well as $(\alpha(q_{0}),\alpha_{\sup}),$ is linear and equals $T(q_{1})+q_{1}\alpha$ on $(\alpha(q_{1}),\alpha_{\lim}),$ is linear and equals $T(q_{0})+q_{0}\alpha$ on $(\alpha_{\lim},\alpha(q_{0})),$ and has its maximum at $\alpha(0).$  Then, the multfractal spectrum has phase transitions at $\alpha(q_{0}), \alpha_{\lim}, \text{ and } \alpha(q_{1})$ in this case.
\end{enumerate}
In fact, Proposition~\ref{prop:PCD1} gives these results. 

\subsection{The Case $\alpha_{\lim}=\infty$}\hspace*{\fill} \par
We provide the following example in which $\alpha_{\lim}=\infty.$  Let $\Sigma=\mathbb{N}^{\mathbb{N}}.$  Define $\phi:\Sigma \rightarrow \mathbb{R}^{-}$ and $\psi:\Sigma \rightarrow \mathbb{R}^{+}$ as follows.  For each $i \in \mathbb{N},$ let
$p_{i}=\left(\frac{1}{2}\right)^{i}.$  We remark that $p_{i}$ is exactly the Minkowski $?-$function.  Clearly, $\sum_{i=1}^{\infty}p_{i}=1.$  Again, we approximate $\psi(x)=\log|G'(\pi(x))|$ with $\bar{\psi}(x)=-\log\left(\frac{6}{\pi^{2}x_{1}^{2}}\right)$ and we let
\[
\phi(x)=\bar{\phi}(x)=\log{p_{x_{1}}}=\log\left(\left(\frac{1}{2}\right)^{x_{1}}\right).
\]
for each $x \in \Sigma.$ 

Let $\bar{j}=(j,j,j,...)$ such that $j \in \mathbb{N}.$  Since $\bar{\phi}$ and $\bar{\psi}$ are locally constant, we can approximate $\mathcal{P}(q\phi-t\psi)$ with
\[\mathcal{P}(q\bar{\phi}-t\bar{\psi})=\log{Z_{1}(q\bar{\phi}-t\bar{\psi})}=\log\left(\sum_{j=1}^{\infty}\frac{(\frac{6}{\pi^{2}})^{t}}{2^{jq}(j)^{2t}}\right).\]
We notice that
\begin{enumerate}
\item For each $q<0,$ \[ \log\left(\sum_{j=1}^{\infty}\frac{(\frac{6}{\pi^{2}})^{t}}{2^{jq}(j)^{2t}}\right)=\infty\]
independent of the choice of $t\in\mathbb{R}.$  Then, $T(q)=\tilde{t}(q)=\infty$ for these $q.$
\item If $q=0,$
\[\log\left(\sum_{j=1}^{\infty}\frac{(\frac{6}{\pi^{2}})^{t}}{(j)^{2t}}\right)<\infty\]
for every $t>\frac{1}{2}\in\mathbb{R}.$  Hence, $\tilde{t}(0)=\frac{1}{2}<T(0).$  In fact, $T(0)=1.$
\item For each $q>0,$ \[\log\left(\sum_{j=1}^{\infty}\frac{(\frac{6}{\pi^{2}})^{t}}{2^{jq}(j)^{2t}}\right)<\infty\]
independent of the choice of $t\in\mathbb{R}.$  Then, $-\infty=\tilde{t}(q)<T(q)$ for these $q.$
\end{enumerate}
\par By Lemma~\ref{lemma:MU1}, we get that
$$\tilde{t}(q) =
\left\{
	\begin{array}{ll}
		\infty  & \mbox{if } q < 0 \\
		\frac{1}{2} & \mbox{if } q = 0 \\
        -\infty & \mbox{if } q > 0.
	\end{array}
\right.$$
It follows that 
$$Q=\{q \in \mathbb{R}: T(q)=\tilde{t}(q)\}=(-\infty,0).$$
By convention,
$$\alpha(0)=\lim\limits_{q \rightarrow 0^{+}}\alpha(q).$$
Also, notice that
$\alpha_{\lim}=\lim\limits_{i \rightarrow \infty}\frac{\log{p_{i}}}{\log{s_{i}}}=\infty.$  Then, $\alpha(0)<\alpha_{\lim}=\infty=\alpha_{\sup}$ because $T(0)=1>\tilde{t}(0)$ gives us that \[\lim\limits_{q \rightarrow 0^{+}}\int-\phi\, \mathrm{d}\mu_{q} \le \lim\limits_{q \rightarrow 0^{+}}\sum_{j=1}^{\infty}\frac{(\frac{6}{\pi^{2}})^{T(q)}j\log(2)}{2^{jq}(j)^{2T(q)}} \in (-\infty,\infty)\] \[ \text{and} \lim\limits_{q \rightarrow 0^{+}}\int\psi\, \mathrm{d}\mu_{q}\le \lim\limits_{q \rightarrow 0^{+}}\sum_{j=1}^{\infty}\frac{(\frac{6}{\pi^{2}})^{T(q)}[2\log(j)-\log(\frac{6}{\pi^{2}})]}{2^{jq}(j)^{2T(q)}} \in (-\infty,\infty).\]
Hence, we have the following analysis for the behaviour of the multifractal spectrum.

\begin{proposition}
For this example, the multifractal spectrum is increasing and analytic on $(\alpha_{\inf},\alpha(0)).$  $f_{\mu}(\alpha)=T(0)$ on $(\alpha(0),\infty).$  Furthermore, the multifractal spectrum has a phase transition at $\alpha(0).$
\end{proposition}
\begin{proof}
Each $\alpha \in (\alpha_{\inf},\alpha(0))$ satisfies $\alpha=\alpha(q)$ for a unique $q>0.$  Since each of these $q$ are in $Q^{\complement},$ $f_{\mu}(\alpha)$ is analytic on $(\alpha_{\inf},\alpha(0)).$  The multifractal spectrum increases on $(\alpha_{\inf},\alpha(0))$ follows from Proposition~\ref{prop:ID1}.  For $\alpha \in (\alpha(0),\alpha_{\lim}),$ $f_{\mu}(\alpha)=T(0).$  Hence, the multifractal spectrum is constant on $(\alpha(0),\alpha_{\lim}).$
\end{proof}

\subsection{The Case When $\alpha_{\lim}$ Does Not Exist}\hspace*{\fill} \par
When $\alpha_{\lim}$ does not exist, the multifractal spectrum has up to infinitely many phase transitions.  We remark that Iommi and Jordan \cite{iommi2013phase} create a similar example in the setting of a suspension flow.  Now, we roughly outline the procedure for creating such an example in our setting. First, we take the locally H\"older potentials $\phi$ and $\psi$ such that
\[\phi(x)=\log{p_{x_{1}}} \text{ and } \psi(x)=\log|G'(\pi(x))|\]
for each $x=(x_{1},x_{2},...)\in\Sigma.$
Then, to define the $p_{i}$ for each $i \in \mathbb{N},$ we partition the natural numbers as follows.  Let $r_{0}=0$ and $r_{1}=1.$  Consider the infinite sequence $\{r_{k}\}_{k=2}^{\infty}$ of primes $\{2,3,5,7,11,...\}.$  We define the sets $\{I_{k}\}_{k \in \mathbb{N}_{0}}$ as follows:
\[I_{0}:=\{m \in \mathbb{N} \text{ such that } m \text{ cannot be written as any prime power of any } n \in \mathbb{N}\}.\]
\[I_{1}:=\{m \in \mathbb{N} \text{ that can be written as the } 2\text{nd power of some } n \in \mathbb{N} \}\]
In general,
\[I_{k}:=\{m \in \mathbb{N} \text{ that can be written as the } r_{k+1}\text{st power of some } n \in \mathbb{N} \}.\]

For each $m \in \mathbb{N},$ we get that
\[p_{m} =\frac{C_{k}}{m^{l_{k}}(\log(m+2))^{M_{k}}}\] 
if $k \in\mathbb{N}_{0}$ and $m \in I_{k}.$  We have increasing sequences of constants $\{C_{k}\}_{k \in \mathbb{N}_{0}}$ and $\{M_{k}\}_{k \in \mathbb{N}_{0}}.$  The terms of both sequences are chosen such that $\sum_{m=1}^{\infty}p_{m}=1.$  For each $k \in \mathbb{N} \cup \{0\},$ we have a recursive relation for the values of the sequence $\{l_{k}\}$ stated in the expression for $p_{m}.$
We remind the reader that locally H\"older potentials can be approximated by locally constant potentials.  For each $a_{1}(x)=m \in \mathbb{N},$ define
\[s_{m}:= \frac{6}{\pi^{2}m^{2}}\] for each $m \in \mathbb{N};$ hence, we can estimate $\psi$ with $\bar{\psi}(x)=\log{s_{x_{1}}^{-1}}.$ \par 

For each $I_{k},$ we have a function $\tilde{t}_{k}(q)$ as follows:
\[\tilde{t}_{k}(q):=\inf\{t \in \mathbb{R}: \sum_{I_{k}}p_{m}^{q}s_{m}^{\tilde{t}_{k}(q)}<\infty\}.\]
We define
\[\tilde{t}(q)=\left\{\sup_{q \in \mathbb{R}}\tilde{t}_{k}(q): k\in \mathbb{N} \cup \{0\}\right\}.\]
By construction, $|\tilde{t}_{k}'|>|\tilde{t}_{k+1}'|$ and each $\tilde{t}_{k}$ is linear.  We get that the phase transitions for $\tilde{t}(q)$ occur at values of $q$ such that $\tilde{t}_{k}(q)=\tilde{t}_{k+1}(q).$  Proceeding with the computation of these $q,$ we find that they occur at each $q \in \mathbb{N}.$  Hence, $\tilde{t}(q)$ has infinitely many phase transitions.

Finally, we prove that $Z_{1}(q\bar{\phi}-\tilde{t}(q)\bar{\psi})<1$
for each $q \in (k,k+1) \cup (a,1)$ for some $a \in \mathbb{R}$ and for each $k \in \mathbb{N}.$  This gives us that $\mathcal{P}(q\bar{\phi}-\tilde{t}(q)\bar{\psi})<1$
for those $q;$ hence, $T(q)=\tilde{t}(q)$ for all $q \ge \bar{q} \ge a$ (for some $1> \bar{q} \ge a$).  Using techniques from the previous example, we get results for a possible phase transition for the multifractal spectrum at $\alpha(\bar{q}).$  Without loss of generality, let us assume that there is no phase transition at $\alpha(\bar{q}).$  Hence, we have the following behaviour.  $f_{\mu}(\alpha)$ is analytic on $(\alpha(0),\alpha_{\sup}),$ $(\alpha(1),\alpha(0)),$ $(\alpha(2),\alpha(1)),$...,$(\alpha(N),\alpha(N-1)),$....  The phase transitions for $f_{\mu}(\alpha)$ are at $\alpha(1),$ $\alpha(2),$...$,\alpha(N),$....  The multifractal spectrum increases and is piecewise linear on $(\alpha_{\inf},\alpha(N)),$..., equals $T(3)+3\alpha$ and increases on $(\alpha(3),\alpha(2)),$ equals $T(2)+2\alpha$ and increases on $(\alpha(2),\alpha(1)),$ equals $\alpha$ and increases on $(\alpha(1),\alpha(0)),$ and finally, decreases on $(\alpha(0),\alpha_{\sup}).$  Finally, we remark that the existence of $\alpha_{\lim}$ is absolutely necessary for Theorems~\ref{thm:M1} and \ref{thm:M2} to be true.

\end{document}